\documentclass[12pt]{article}

\usepackage{amssymb}
\usepackage{latexsym}
\usepackage{amsmath}
\usepackage{indentfirst}
\usepackage{graphicx}
\usepackage{subfigure}
\usepackage{float}
\usepackage{color}
\usepackage{hyperref}
\usepackage{orcidlink}
\usepackage[compress]{cite}

\title{Size Ramsey minimal graphs for uniform star forests}
\author{Pingting Fu \thanks{School of Mathematics and Statistics, Hainan University, Haikou, Hainan 570028, P. R. China. Email: fpt\_inya@163.com}
\and Zhidan Luo \orcidlink{0009-0005-6195-147X} \thanks{Corresponding author. School of Mathematics and Statistics, Hainan University, Haikou, Hainan 570028, P. R. China. Research supported in part by National Natural Science Foundation of China (No.12401449), Hainan Provincial Natural Science Foundation of China (No.125QN209) and Hainan University Research Foundation Project (No. KYQD(ZR)-23155). Email: luodan@hainanu.edu.cn}
\and Zhenyu Ni \thanks{School of Mathematics and Statistics, Hainan University, Haikou, Hainan 570028, P. R. China. Email: 995264@hainanu.edu.cn}}
\date{}

\newtheorem{theo}{Theorem}[section]

\newtheorem{remark}[theo]{Remark}
\newtheorem{lemma}[theo]{Lemma}

\newtheorem{fact}[theo]{Fact}

\newtheorem{conj}[theo]{Conjecture}

\def\q{\hspace*{\fill}$\Box$\medskip}

\begin{document}
\maketitle
\begin{abstract}
  For given graphs $G_{1}, G_{2}, \dots, G_{t}$ and $G$, let $G\rightarrow (G_{1}, G_{2}, \dots, G_{t})$ denote that each $t$-coloring of $E(G)$ yields a monochromatic copy of $G_{i}$ in color $i$ for some $i\in [t]$. The {\it size Ramsey number}, $\hat{r}(G_{1}, G_{2}, \dots, G_{t})$ is the minimum size of $G$ such that $G\rightarrow (G_{1}, G_{2}, \dots, G_{t})$. A graph $G$ is a {\it size Ramsey minimal graph} for $(G_{1}, G_{2}, \dots, G_{t})$ if $G\rightarrow (G_{1}, G_{2}, \dots, G_{t})$ and $e(G)= \hat{r}(G_{1}, G_{2}, \dots, G_{t})$. A {\it star forest} is a vertex-disjoint union of stars, and a {\it uniform star forest} is a star forest with the same size of each component. In 1978, Burr, Erd\H{o}s, Faudree, Rousseau and Schelp, and in 2025, Davoodi, Javadi, Kamranian and Raeisi completely characterized the size Ramsey minimal graphs for uniform star forests. In this paper, we completely characterize the size Ramsey minimal graphs for uniform star forests in multicolors.
  
  
\end{abstract}

\section{Introduction}
  In this paper, all graphs are simple. Moreover, we ignore isolated vertices. Let $V(G)$ and $E(G)$ be the vertex set and the edge set of $G$, respectively. The {\it size} of $G$ is $|E(G)|$ and the {\it order} of $G$ is $|V(G)|$. Denote them by $e(G)$ and $v(G)$, respectively. For given graphs $G_{1}, G_{2}, \dots, G_{t}$ and $G$, let $G\rightarrow (G_{1}, G_{2}, \dots, G_{t})$ denote that each $t$-coloring of $E(G)$ yields a monochromatic copy of $G_{i}$ in color $i$ for some $i\in [t]$. The {\it size Ramsey number}, $\hat{r}(G_{1}, G_{2}, \dots, G_{t})$ is the minimum size of $G$ such that $G\rightarrow (G_{1}, G_{2}, \dots, G_{t})$, that is
  $$\hat{r}(G_{1}, G_{2}, \dots, G_{t})= \min\{e(G): G\rightarrow (G_{1}, G_{2}, \dots, G_{t})\}.$$
  If $G\rightarrow (G_{1}, G_{2}, \dots, G_{t})$ and $e(G)= \hat{r}(G_{1}, G_{2}, \dots, G_{t})$, then we call $G$ a {\it size Ramsey minimal graph} for $(G_{1}, G_{2}, \dots, G_{t})$. For results concerning size Ramsey number, we refer the reader to \cite{bls,nk,ErdosFaudreeRousseauSchelp,FaudreeSchelp,JavadiOmidi,Pikhurko,k} and references therein.
  
  For graphs $G$ and $H$, let $G\sqcup H$ be the vertex-disjoint union of $G$ and $H$, and let $tG$ be the vertex-disjoint union of $t$ copies of $G$. For given positive integers $n, n_{1}, n_{2}, \dots, n_{t}$, let $K_{1, n}$, $\bigsqcup_{i= 1}^{t} K_{1, n_{i}}$ and $\bigsqcup_{i= 1}^{t} K_{1, n}$ be the {\it star}, the {\it star forest} and the {\it uniform star forest}, respectively. In 1978, Burr, Erd\H{o}s, Faudree, Rousseau and Schelp considered the size Ramsey number for star forests and conjectured the following.
  
  \begin{conj}[Burr, Erd\H{o}s, Faudree, Rousseau, Schelp \cite{BurrErdosFaudreeRousseauSchelp}]\label{conj1.1}
    For given positive integers $s$ and $t$, let $n_{1}\geq n_{2}\geq \cdots\geq n_{s}\geq 1$ and $m_{1}\geq m_{2}\geq \cdots\geq m_{t}\geq 1$ be integers. For each $k\in [s+ t]\backslash [1]$, let $\ell_{k}= \max\left\{n_{i}+ m_{j}- 1: i+ j= k\right\}$. Then 
    $$\hat{r}\left(\bigsqcup_{i= 1}^{s} K_{1, n_{i}}, \bigsqcup_{j= 1}^{t} K_{1, m_{j}}\right)= \sum_{k= 2}^{s+ t} \ell_{k}.$$
  \end{conj}
  
  \noindent In the same paper, they confirmed Conjecture \ref{conj1.1} for uniform star forests. Moreover, they characterized the size Ramsey minimal graphs for uniform star forests. 
  
  \begin{theo}[Burr, Erd\H{o}s, Faudree, Rousseau, Schelp \cite{BurrErdosFaudreeRousseauSchelp}]\label{theo1.2}
    For given positive integers $s$, $t$, $m$ and $n$, $\hat{r}(sK_{1, n}, tK_{1, m})= (s+ t- 1)(m+ n- 1)$. If $G$ is a size Ramsey minimal graph for $(sK_{1, n}, tK_{1, m})$, then $G= (s+ t- 1)K_{1, m+ n- 1}$. Moreover, if $m= n= 2$, then also $G= cK_{3}\sqcup (s+ t- c- 1)K_{1, 3}$ for some $c\in [s+ t- 1]$. 
  \end{theo}
  
  In 2002, Gy\H{o}ri and Schelp \cite{GyoriSchelp} confirmed Conjecture \ref{conj1.1} under the condition ${\ell_{i}\choose 2}\geq \sum_{k= i}^{s+ t} \ell_{k}$ for each $i\in [s+ t]\backslash [1]$. After that, Conjecture \ref{conj1.1} has no progress until 2025. Davoodi, Javadi, Kamranian and Raeisi \cite{Davoodi} confirmed Conjecture \ref{conj1.1} for several cases ($n_{i}$ and $m_{j}$ are odd, or $s= 1$, and more), and completely characterized the size Ramsey minimal graphs for $(sK_{1, n}, tK_{1, m})$ since there is a missing case in Theorem \ref{theo1.2}. 
  
  \begin{theo}[Davoodi, Javadi, Kamranian and Raeisi \cite{Davoodi}]\label{theo1.3}
    If $G$ is a size Ramsey minimal graph for $(sK_{1, n}, tK_{1, m})$, then $G= (s+ t- 1)K_{1, m+ n- 1}$. Moreover, if $m= n= 2$, then also $G= cK_{3}\sqcup (s+ t- c- 1)K_{1, 3}$ for some $c\in [s+ t- 1]$; if $s= 1, n= 2$ and $m= 1$, then also $G= cC_{4}\sqcup (t- 2c)K_{1, 2}$ for some $c\in [\lfloor t/2\rfloor]$.
  \end{theo}
  
  In fact, earlier than 2025, Zhang extended Theorem \ref{theo1.2} to multicolors.
  
  \begin{theo}[Zhang \cite{Zhang}]\label{theo1.4}
    For a given positive integer $t$, let $a_{1}, a_{2}, \dots, a_{t}$ and $b_{1}, b_{2}, \cdots, b_{t}$ be positive integers. Then
    $$\hat{r}\left(a_{1}K_{1, b_{1}}, a_{2}K_{1, b_{2}}, \dots, a_{t} K_{1, b_{t}}\right)= \left(\sum_{s= 1}^{t} a_{s}- t+ 1\right)\left(\sum_{s= 1}^{t} b_{s}- t+ 1\right).$$
  \end{theo}
  
  In this paper, we completely characterize the size Ramsey minimal graphs for uniform star forests in multicolors.
  
  \begin{theo}\label{theo1.5}
    For a given positive integer $t$, let $a_{1}, a_{2}, \dots, a_{t}$ be positive integers and $b_{1}\geq b_{2}\geq \cdots\geq b_{t}\geq 1$ be integers. If $G$ is a size Ramsey minimal graph for $\left(a_{1}K_{1, b_{1}}, a_{2}K_{1, b_{2}}, \dots, a_{t}K_{1, b_{t}}\right)$, then $G= (a- t+ 1)K_{1, b}$, where $a= \sum_{s= 1}^{t} a_{s}$ and $b= \sum_{s= 1}^{t} b_{s}- t+ 1$. Moreover, the following holds.
    
    {\rm (1)} If $a_{1}= 1, b_{1}= 2$ and $b_{2}= 1$, then also $G= cC_{4}\bigsqcup (a- t+ 1- 2c)K_{1, 2}$, where $c\in [\lfloor (a- t+ 1)/2\rfloor]$.

    {\rm (2)} If $b_{1}= b_{2}= 2$ and $b_{3}= 1$, then also $G= cK_{3}\bigsqcup (a- t+ 1- c)K_{1, 3}$, where $c\in [a- t+ 1]$. Moreover, if $a_{1}= a_{2}= 1$, then also $G= cK_{3}\bigsqcup c'K_{4}\bigsqcup (a- t+ 1- c- 2c')K_{1, 3}$, where $c$ and $c'$ are nonnegative integers such that $1\leq c+ 2c'\leq a- t+ 1$.
  \end{theo}
  
  \begin{remark}\label{remark1.6}
    Let $a_{3}= a_{4}= \cdots= a_{t}= 1$ and $b_{3}= b_{4}= \cdots= b_{t}= 1$, and then Theorem \ref{theo1.5} is Theorem \ref{theo1.3} since a graph is the size Ramsey minimal graph for $(a_{1}K_{1, b_{1}}, a_{2}K_{1, b_{2}})$ if and only if it is the size Ramsey minimal graph for $(a_{1}K_{1, b_{1}}, a_{2}K_{1, b_{2}}, K_{1, 1}, K_{1, 1}, \dots, K_{1, 1})$.
  \end{remark}

\section{Preliminaries}
  For a graph $G$ and a vertex $v\in V(G)$, let $N_{G}(v)$ and $d_{G}(v)$ be the neighbours and the degree of $v$ in $G$, respectively. If the graph $G$ is unique, then simplify them as $N(v)$ and $d(v)$. Moreover, let $N[v]= N(v)\cup \{v\}$ and let $\Delta(G)$ be the maximum degree of $G$. For vertex sets $U, W\subset V(G)$ such that $U\cap W= \emptyset$, let $G[U]$ and $G[U, W]$ be the graph induced by $G$ on $U$ and induced by edges of $G$ between $U$ and $W$, respectively. Let $G- U$ be the graph obtained from $G$ by removing the vertex set $U$ and edges adjacent to vertex in $U$. For graphs $G$ and $H$, let $G\cup H$ be the union of $G$ and $H$. Let $K_{n}$, $P_{n}$ and $C_{n}$ be the complete graph, the path and the cycle of order $n$, respectively. A {\it proper edge coloring} of $G$ is a coloring of $E(G)$ such that incident edges receive distinct colors. The {\it edge chromatic number} of $G$, $\chi'(G)$ is the minimum number of colors over all proper edge coloring of $G$. A {\it component} of $G$ is a maximal connected subgraph of $G$. The {\it center} of $K_{1, n}$ is the vertex with degree $n$. For a given positive integer $t$, let $a_{1}, a_{2}, \dots, a_{t}, b_{1}, b_{2}, \dots, b_{t}$ be positive integers.
  
  \begin{fact}\label{fact2.1}
    For given positive integers $n_{1}, n_{2}, \dots, n_{t}$, let $G$ be a graph. If $\Delta(G)\leq \sum_{s= 1}^{t} n_{s}- t- 1$, then $G\not\rightarrow (K_{1, n_{1}}, K_{1, n_{2}}, \dots, K_{1, n_{t}})$.
  \end{fact}
  \begin{proof}
    Note that $\chi'(G)\leq \Delta(G)+ 1\leq \sum_{s= 1}^{t} n_{s}- t$ by the Vizing Theorem \cite{VZ}. Thus, there is a proper edge coloring of $G$ with $\sum_{s= 1}^{t} n_{s}- t$ colors. In other words, $E(G)$ can be decomposed into $\sum_{s= 1}^{t} n_{s}- t$ edge-disjoint matchings. Color $n_{s}- 1$ matchings by color $s$ for each $s\in [t]$. Consequently, $G\not\rightarrow \left(K_{1, n_{1}}, K_{1, n_{2}}, \dots, K_{1, n_{t}}\right)$.\q
  \end{proof}

  We also need the following on induction.
    
  \begin{lemma}\label{lemma2.2}
    Let $G$ be a size Ramsey minimal graph for $(a_{1}K_{1, b_{1}}, a_{2}K_{1, b_{2}}, \dots, a_{t}K_{1, b_{t}})$ and let $v\in V(G)$ be a vertex. For each $s\in [t]$, if $a_{s}\geq 2$, then 
    $$G- \{v\}\rightarrow \left(a_{1}K_{1, b_{1}}, \dots, a_{s- 1}K_{1, b_{s- 1}}, (a_{s}- 1)K_{1, b_{s}}, a_{s+ 1}K_{1, b_{s+ 1}}, \dots, a_{t} K_{1, b_{t}}\right).$$
  \end{lemma}
  \begin{proof}
    Otherwise, there is a $t$-coloring of $E(G- \{v\})$ such that there is no monochromatic copy of $a_{j}K_{1, b_{j}}$ in color $j$ for each $j\in [t]\backslash \{s\}$ and no monochromatic copy of $(a_{s}- 1) K_{1, b_{s}}$ in color $s$. Under this coloring of $G- \{v\}$, color all edges adjacent to $v$ by color $s$. To create a monochromatic copy of $a_{s} K_{1, b_{s}}$ in color $s$, we need at least two vertices that are not in $G- \{v\}$. Consequently, $G\not\rightarrow \left(a_{1}K_{1, b_{1}}, a_{2}K_{1, b_{2}}, \dots, a_{t} K_{1, b_{t}}\right)$. This is a contradiction.\q
  \end{proof}
  
  For convenience to describe the coloring in Section \ref{sec3}, we need the following facts.
  
  \begin{fact}\label{fact2.3}
    Let $c_{1}, c_{2}, \dots, c_{t}$ and $c$ be positive integers. Let $G$ be a graph, and let $E\subset E(G)$ with $|E|= c$. If $G- E\not\rightarrow (c_{1}K_{1, b_{1}}, c_{2}K_{1, b_{2}}, \dots, c_{t}K_{1, b_{t}})$, then 
    $$G\not\rightarrow ((c_{1}+ p_{1})K_{1, b_{1}}, (c_{2}+ p_{2})K_{1, b_{2}}, \dots, (c_{t}+ p_{t})K_{1, b_{t}}),$$
    where $p_{s}$ is a nonnegative integer for each $s\in [t]$ and $\sum_{s= 1}^{t} p_{s}= c$.
  \end{fact}
  \begin{proof}
    Note that an edge in color $s$ creates at most a monochromatic copy of $K_{1, b_{s}}$ in color $s$. Color $p_{s}$ edges of $E$ by color $s$ for each $s\in [t]$. Thus, 
    $$G\not\rightarrow ((c_{1}+ p_{1})K_{1, b_{1}}, (c_{2}+ p_{2})K_{1, b_{2}}, \dots, (c_{t}+ p_{t})K_{1, b_{t}})$$ 
    since $G- E\not\rightarrow (c_{1}K_{1, b_{1}}, c_{2}K_{1, b_{2}}, \dots, c_{t}K_{1, b_{t}})$.\q
  \end{proof}
  
  \begin{fact}\label{fact2.4}
    Let $c, c_{1}, c_{2}, \dots, c_{t}$ and $b$ be positive integers. Let $G$ be a graph without isolated vertex. Let $H= cK_{1, b}$ and let $V$ be the centers of $H$. If $V\cap V(G)= \emptyset$ and $G\not\rightarrow (c_{1}K_{1, b_{1}}, c_{2}K_{1, b_{2}}, \dots, c_{t}K_{1, b_{t}})$, then 
    $$G\cup H\not\rightarrow ((c_{1}+ p_{1})K_{1, b_{1}}, (c_{2}+ p_{2})K_{1, b_{2}}, \dots, (c_{t}+ p_{t})K_{1, b_{t}}),$$
    where $p_{s}$ is a nonnegative integer for each $s\in [t]$ and $\sum_{s= 1}^{t} p_{s}= c$.
  \end{fact}
  \begin{proof}
    Note that a monochromatic copy of $K_{1, n_{s}}$ in color $s$, which is not in $G$, contains at least one edge in color $s$ of $H$, and thus contains a center with edges in color $s$ since $H= cK_{1, b}$. Color $p_{s}$ copies of $K_{1, b}$ of $H$ by color $s$ for each $s\in [t]$. Thus, 
    $$G\cup H\not\rightarrow ((c_{1}+ p_{1})K_{1, b_{1}}, (c_{2}+ p_{2})K_{1, b_{2}}, \dots, (c_{t}+ p_{t})K_{1, b_{t}})$$
    since $G\not\rightarrow (c_{1}K_{1, b_{1}}, c_{2}K_{1, b_{2}}, \dots, c_{t}K_{1, b_{t}})$.\q
  \end{proof}
  
\section{Size Ramsey minimal graphs for uniform star forests}\label{sec3}

  In this section, we characterize the size Ramsey minimal graphs for uniform star forests. Let $G$ be a size Ramsey minimal graph for $(a_{1}K_{1, b_{1}}, a_{2}K_{1, b_{2}}, \dots, a_{t}K_{1, b_{t}})$, and then $e(G)= (a- t+ 1)b$ by Theorem \ref{theo1.4}, where $a= \sum_{s= 1}^{t} a_{s}$ and $b= \sum_{s= 1}^{t} b_{s}- t+ 1$. We firstly show that $\Delta(G)$ is bounded.
  
  \begin{fact}\label{fact3.1}
    $b- 1\leq \Delta(G)\leq b$.
  \end{fact}
  \begin{proof}
    If $\Delta(G)\leq b- 2= \sum_{s= 1}^{t} b_{s}- t- 1$, then $G\not\rightarrow (K_{1, b_{1}}, K_{1, b_{2}}, \dots, K_{1, b_{t}})$ by Fact \ref{fact2.1}, and thus $G\not\rightarrow (a_{1}K_{1, b_{1}}, a_{2}K_{1, b_{2}}, \dots, a_{t}K_{1, b_{t}})$ since $K_{1, b_{s}}\subset a_{s}K_{1, b_{s}}$ for each $s\in [t]$. This is a contradiction. Consequently, $\Delta(G)\geq b- 1$.
    
    If $a= t$ (that is, $a_{1}= a_{2}= \cdots= a_{t}= 1$), then $e(G)= b$. Thus, $\Delta(G)\leq b$, and we are done. Suppose that $a\geq t+ 1$, and then $a_{s}\geq 2$ for some $s\in [t]$. Let $v\in V(G)$ be such that $d(v)= \Delta(G)$. By Lemma \ref{lemma2.2}, 
    $$G- \{v\}\rightarrow \left(a_{1}K_{1, b_{1}}, \dots, a_{s- 1}K_{1, b_{s- 1}}, (a_{s}- 1)K_{1, b_{s}}, a_{s+ 1}K_{1, b_{s+ 1}}, \dots, a_{t} K_{1, b_{t}}\right).$$
    Thus, $e(G- \{v\})\geq (a- t)b$ by Theorem \ref{theo1.4}. Furthermore,
    $$(a- t+ 1)b= e(G)= d(v)+ e(G- \{v\})\geq d(v)+ (a- t)b.$$
    Consequently, $d(v)\leq b$, and we are done.\q
  \end{proof}
  
  Now, we are ready to proof Theorem \ref{theo1.5}. By Remark \ref{remark1.6}, we may assume that $t\geq 3$. In the following, we divide Theorem \ref{theo1.5} into Lemma \ref{lemma3.2}, Theorem \ref{theo3.3}, Theorem \ref{theo3.4}, Theorem \ref{theo3.5} and Theorem \ref{theo3.6}.
  
  \begin{lemma}\label{lemma3.2}
    Suppose that $b_{1}\geq b_{2}\geq \cdots \geq b_{t}\geq 1$. If $G$ is a size Ramsey minimal graph for $\left(K_{1, b_{1}}, K_{1, b_{2}}, \dots, K_{1, b_{t}}\right)$, then $G= K_{1, b}$, where $b= \sum_{s= 1}^{t} b_{s}- t+ 1$. Moreover, if $b_{1}= b_{2}= 2$ and $b_{3}= 1$, then also $G= K_{3}$.
  \end{lemma}
  \begin{proof} 
    Note that $e(G)= b$ by Theorem \ref{theo1.4}. If $b= 1$, then $b_{1}= b_{2}= \cdots= b_{t}= 1$ and $e(G)= 1$. Thus, $G= K_{1, 1}$, and we are done. Suppose that $b\geq 2$, and thus $b_{1}\geq 2$ by our assumption. Let $C_{1}, C_{2}, \dots, C_{c}$ be components (not isolated vertex) of $G$, and then $c\geq 1$ since $e(G)= b\geq 2$. Moreover, suppose that $e(C_{1})\geq e(C_{2})\geq \cdots\geq e(C_{c})$. By Fact \ref{fact3.1}, $\Delta(G)\geq b- 1$, and thus $e(C_{1})\geq b- 1$. Furthermore, $c\leq 2$ since $e(G)= b$.
    
    If $c= 2$, then $C_{2}= K_{1, 1}$ and $C_{1}= K_{1, b- 1}$ since $\Delta(G)\geq b- 1$. Color the edge of $C_{2}$ by color $1$, and color $b_{s}- 1$ edges of $C_{1}$ by color $s$ for each $s\in [t]$. Thus, $G\not\rightarrow \left(K_{1, b_{1}}, K_{1, b_{2}}, \dots, K_{1, b_{t}}\right)$. This is a contradiction. 
    
    Consequently, $c= 1$, and thus $G$ is connected. We divide the discussion into two parts since $b- 1\leq \Delta(G)\leq b$. If $\Delta(G)= b$, then $G= K_{1, b}$ since $e(G)= b$, and we are done. Suppose that $\Delta(G)= b- 1$. Note that $G$ is the union of $K_{1, b- 1}$ and $K_{1, 1}$ (they have at least one common vertex). If $b_{1}\geq 3$, then color the edge not in $K_{1, b- 1}$ by color $1$, and color $b_{s}- 1$ edges of $K_{1, b- 1}$ by color $s$ for each $s\in [t]$. Thus, $G\not\rightarrow \left(K_{1, b_{1}}, K_{1, b_{2}}, \dots, K_{1, b_{t}}\right)$ since $b_{1}\geq 3$. This is a contradiction.
    
    Consequently, $b_{1}= 2$. Suppose that $2= b_{1}= b_{2}= \cdots= b_{q}> b_{q+ 1}= b_{q+ 2}= \cdots= b_{t}= 1$ for some $q\in [t]$. In this case, $b= q+ 1$. If $q\geq 3$, then $b= q+ 1\geq 4$. Thus, $G$ contains a copy of $2K_{1, 1}$. Color the copy of $2K_{1, 1}$ by color $1$. After that, there are $b- 2= q- 1$ uncolored edges. Color one of them by color $s$ for each $s\in [q]\backslash [1]$. Thus, $G\not\rightarrow \left(K_{1, b_{1}}, K_{1, b_{2}}, \dots, K_{1, b_{t}}\right)$. This is a contradiction.
    
    Consequently, $q\leq 2$. If $q= 1$, then $e(G)= b= q+ 1= 2$. It is impossible since $G$ is connected and $\Delta(G)= b- 1= 1$. Consequently, $q= 2$, and thus $b= q+ 1= 3$. Therefore, $G= P_{4}$ or $G= K_{3}$. Note that $P_{4}\not\rightarrow (K_{1, 2}, K_{1, 2}, K_{1, 1}, \dots, K_{1, 1})$. Consequently, $G= K_{3}$, and we are done.\q
  \end{proof}
  
  \begin{theo}\label{theo3.3}
    Suppose that $a_{1}\geq a_{2}\geq \cdots\geq a_{t}\geq 1$. If $G$ is a size Ramsey minimal graph for $\left(a_{1}K_{1, 1}, a_{2}K_{1, 1}, \dots, a_{t}K_{1, 1}\right)$, then $G= (a- t+ 1)K_{1, 1}$, where $a= \sum_{s= 1}^{t} a_{s}$.
  \end{theo}
  \begin{proof}
    The assertion holds for $a= t$ by Lemma \ref{lemma3.2}. Suppose that $a\geq t+ 1$, and thus $a_{1}\geq 2$. Note that $e(G)= a- t+ 1$ by Theorem \ref{theo1.4}. Let $C_{1}, C_{2}, \cdots, C_{c}$ be components of $G$, and then $c\geq 1$. Moreover, suppose that $e(C_{1})\geq e(C_{2})\geq \cdots\geq e(C_{c})$. If $e(C_{1})\geq 2$, then color two incident edges in $C_{1}$ by color $1$. Let $G'$ be the graph induced by colored edges. Note that $G'\not\rightarrow (2K_{1, 1}, K_{1, 1}, \dots, K_{1, 1})$. Moreover, there are $a- t+ 1- 2= a- t- 1$ uncolored edges. Thus, $G\not\rightarrow \left(a_{1}K_{1, 1}, a_{2}K_{1, 1}, \dots, a_{t}K_{1, 1}\right)$ by Fact \ref{fact2.3}. This is a contradiction. Consequently, $e(C_{1})= 1$. Furthermore, $G= (a- t+ 1)K_{1, 1}$ since $e(G)= a- t+ 1$, and we are done.\q
  \end{proof}
  
  \begin{theo}\label{theo3.4}
    Suppose that $a_{1}\geq 1$ and $a_{2}\geq a_{3}\geq \cdots\geq a_{t}\geq 1$. If $G$ is a size Ramsey minimal graph for $\left(a_{1}K_{1, 2}, a_{2}K_{1, 1}, a_{3}K_{1, 1}, \dots, a_{t}K_{1, 1}\right)$, then $G= (a- t+ 1)K_{1, 2}$. Moreover, if $a_{1}= 1$, then also $G= cC_{4}\bigsqcup (a- t+ 1- 2c)K_{1, 2}$ for some $c\in [\lfloor (a- t+ 1)/2\rfloor]$.
  \end{theo}
  \begin{proof}
    We use induction in $a$. The assertion holds for $a= t$ by Lemma \ref{lemma3.2}. Suppose that the assertion holds for $a- 1$ and $a\geq t+ 1$. Thus, $a_{1}\geq 2$ or $a_{2}\geq 2$. Note that $e(G)= 2(a- t+ 1)$ by Theorem \ref{theo1.4}, and $1\leq \Delta(G)\leq 2$ by Fact \ref{fact3.1} since $b= 2$. If $\Delta(G)= 1$, then $G= 2(a- t+ 1)K_{1, 1}$. Color all edges of $G$ by color $1$, and thus $G\not\rightarrow \left(a_{1}K_{1, 2}, a_{2}K_{1, 1}, \dots, a_{t}K_{1, 1}\right)$. This is a contradiction. Consequently, $\Delta(G)= 2$. Let $v\in V(G)$ be a vertex such that $d_{G}(v)= \Delta (G)= 2$, and thus $e(G- \{v\})= e(G)- d_{G}(v)= 2(a- t+ 1)- 2= 2(a- t)$. In the following, we divide the discussion into two parts.
    
    {\bf Case 1}. $a_{1}\geq 2$.
    
    Note that $G- \{v\}\rightarrow \left((a_{1}- 1)K_{1, 2}, a_{2}K_{1, 1}, a_{3}K_{1, 1}, \dots, a_{t}K_{1, 1}\right)$ by Lemma \ref{lemma2.2} since $a_{1}\geq 2$. Thus, $G- \{v\}$ is a size Ramsey minimal graph by Theorem \ref{theo1.4}. Furthermore, by the induction hypothesis, $G- \{v\}= (a- t)K_{1, 2}$, or $G- \{v\}= cC_{4}\bigsqcup (a- t- 2c)K_{1, 2}$ if $a_{1}= 2$. Delete all isolated vertices in $G- \{v\}$ and denote the resulting graph by $H_{1}$. 
    
    If $N_{G}(v)\cap V(H_{1})= \emptyset$, then $G= (a- t+ 1)K_{1, 2}$, or $G= cC_{4}\bigsqcup (a- t- 2c+ 1)K_{1, 2}$ if $a_{1}= 2$. We only need to consider the latter. Color a copy of $C_{4}$ by color $1$. Furthermore, color a maximum matching of the uncolored edges by color $1$. Let $G_{1}'$ be the graph induced by colored edges. Note that $G'_{1}\not\rightarrow (2K_{1, 2}, K_{1, 1}, K_{1, 1}, \dots, K_{1, 1})$. Moreover, there are $2(c- 1)+ a- t- 2c+ 1= a- t- 1$ uncolored edges. Thus, $G\not\rightarrow \left(a_{1}K_{1, 2}, a_{2}K_{1, 1}, a_{3}K_{1, 1}, \dots, a_{t}K_{1, 1}\right)$ by Fact \ref{fact2.3}. This is a contradiction. Suppose that $N_{G}(v)\cap V(H_{1})\neq\emptyset$. Note that there are at most two components of $H_{1}$ such that each of them contains at least one vertex of $N_{G}(v)$ since $d_{G}(v)= 2$. 
    
    {\bf Subcase 1.1}. There is only one component of $H_{1}$ such that contains at least one vertex of $N_{G}(v)$.
    
    Denote the vertex set of the component by $A$. Note that $H_{1}[A]= K_{1, 2}$ since $\Delta(G)= 2$. Thus, $H_{1}- A= (a- t- 1)K_{1, 2}$ or $H_{1}- A= cC_{4}\bigsqcup (a- t- 2c- 1)K_{1, 2}$ if $a_{1}= 2$. Moreover, $G[N_{G}[v]\cup A]= P_{5}$ or $G[N_{G}[v]\cup A]= C_{4}$. Color all edges of $G[N_{G}[v]\cup A]$ and a maximum matching of $H_{1}- A$ by color $1$. Let $G_{2}'$ be the graph induced by colored edges. Note that $G'_{2}\not\rightarrow (2K_{1, 2}, K_{1, 1}, K_{1, 1} \dots, K_{1, 1})$. Moreover, there are $a- t- 1$ uncolored edges. Thus, $G\not\rightarrow \left(a_{1}K_{1, 2}, a_{2}K_{1, 1}, a_{3}K_{1, 1}, \dots, a_{t}K_{1, 1}\right)$ by Fact \ref{fact2.3}. This is a contradiction.
    
    {\bf Subcase 1.2}. There are two components of $H_{1}$ such that each of them contains at least one vertex of $N_{G}(v)$.
    
    Denote the vertex set of these components by $B_{1}$ and $B_{2}$, respectively. Moreover, let $B= B_{1}\cup B_{2}$. Note that $H_{1}[B_{1}]= K_{1, 2}$ and $H_{1}[B_{2}]= K_{1, 2}$ since $\Delta(G)= 2$. Thus, $H_{1}- B= (a- t- 2)K_{1, 2}$ or $H_{1}- B= cC_{4}\bigsqcup (a- t- 2c- 2)K_{1, 2}$ if $a_{1}= 2$. Moreover, $G[N_{G}[v]\cup B]= P_{7}$. Note that $a\geq t+ 2$ since $a- t- 2$ is a nonnegative integer, and thus $a_{1}\geq 3$ or $a_{1}= a_{2}= 2, a_{3}= 1$. Color a maximum matching of $H_{1}- B$ by color $1$. If $a_{1}\geq 3$, then color all edges of $G[N_{G}[v]\cup B])$ by color $1$. Let $G_{3}'$ be the graph induced by colored edges. Note that $G'_{3}\not\rightarrow (3K_{1, 2}, K_{1, 1}, K_{1, 1} \dots, K_{1, 1})$. If $a_{1}= a_{2}= 2$ and $a_{3}= 1$, then color a copy of $P_{5}$ of $G[N_{G}[v]\cup B]$ by color $1$ and other edges of $G[N_{G}[v]\cup B]$ by color $2$. Let $G_{4}'$ be the graph induced by colored edges. Note that $G'_{4}\not\rightarrow (2K_{1, 2}, 2K_{1, 1}, K_{1, 1} \dots, K_{1, 1})$. 
    
    In both cases, there are $a- t- 2$ uncolored edges. Thus, $G\not\rightarrow \left(a_{1}K_{1, 2}, a_{2}K_{1, 1}, a_{3}K_{1, 1}, \dots, a_{t}K_{1, 1}\right)$ by Fact \ref{fact2.3}. This is a contradiction.
    
    {\bf Case 2}. $a_{1}= 1$.
    
    In this case, we only need to show that $G= c'C_{4}\bigsqcup (a- t- 2c'+ 1)K_{1, 2}$ for some nonnegative integer $c'$. Note that $a_{2}\geq 2$ by our assumption, and thus 
    $$G- \{v\}\rightarrow \left(K_{1, 2}, (a_{2}- 1)K_{1, 1}, a_{3}K_{1, 1}, a_{4}K_{1, 1}, \dots, a_{t}K_{1, 1}\right)$$
    by Lemma \ref{lemma2.2}. Moreover, $G- \{v\}$ is a size Ramsey minimal graph by Theorem \ref{theo1.4}. Furthermore, by the induction hypothesis, $G- \{v\}= c'C_{4}\bigsqcup (a- t- 2c')K_{1, 2}$. Delete all isolated vertices in $G- \{v\}$ and denote the resulting graph by $H_{2}$. If $N_{G}(v)\cap V(H_{2})= \emptyset$, then $G= c'C_{4}\bigsqcup (a- t- 2c'+ 1)K_{1, 2}$, and we are done. Suppose that $N_{G}(v)\cap V(H_{2})\neq\emptyset$. Note that there are at most two components of $H_{2}$ such that each of them contains at least one vertex of $N_{G}(v)$ since $d_{G}(v)= 2$. 
    
    {\bf Subcase 2.1}. There is only one component of $H_{2}$ such that contains at least one vertex of $N_{G}(v)$.
    
    Denote the vertex set of the component by $C$. Note that $H_{2}[C]= K_{1, 2}$ since $\Delta(G)= 2$. Thus, $H_{2}- C= c'C_{4}\bigsqcup (a- t- 2c'- 1)K_{1, 2}$. Moreover, $G[N_{G}[v]\cup C]= P_{5}$ or $G[N_{G}[v]\cup C]= C_{4}$, and we only need to consider the former. Color a maximum matching of $H_{2}- C$ by color $1$, and color $E_{G}(G[N_{G}[v]\cup C])$ as Figure \ref{f1.1} shows. Let $G_{5}'$ be the graph induced by colored edges. Note that $G'_{5}\not\rightarrow (K_{1, 2}, 2K_{1, 1}, K_{1, 1} \dots, K_{1, 1})$. Moreover, there are $2c'+ a- t- 2c'- 1= a- t- 1$ uncolored edges. Thus, $G\not\rightarrow \left(a_{1}K_{1, 2}, a_{2}K_{1, 1}, a_{3}K_{1, 1}, \dots, a_{t}K_{1, 1}\right)$ by Fact \ref{fact2.3}. This is a contradiction.
     
    \begin{figure}[H]
      \centering
      \includegraphics[scale=0.2]{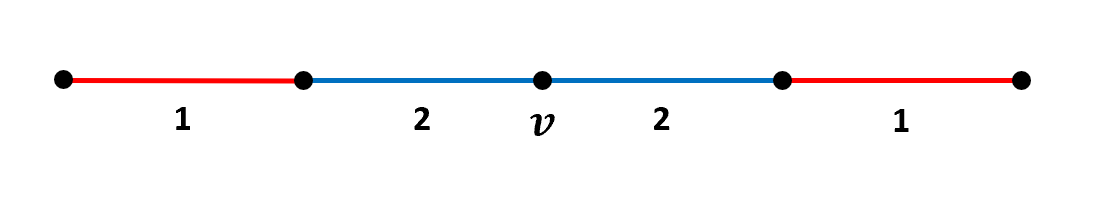}
      \caption{$G[N_{G}(v)\cup C]\not\rightarrow (K_{1, 2}, 2K_{1, 1}, K_{1, 1} \dots, K_{1, 1})$}
      \label{f1.1}
    \end{figure}
    
    {\bf Subcase 1.2}. There are two components of $H_{2}$ such that each of them contains at least one vertex of $N_{G}(v)$.
    
    Denote the vertex set of these components by $D_{1}$ and $D_{2}$, respectively. Moreover, let $D= D_{1}\cup D_{2}$. Note that $H_{2}[D_{1}]= K_{1, 2}$ and $H_{2}[D_{2}]= K_{1, 2}$ since $\Delta(G)= 2$. Thus, $H_{2}- D= c'C_{4}\bigsqcup (a- t- 2c'- 2)K_{1, 2}$. Moreover, $G[N_{G}(v)\cup D]= P_{7}$. Note that $a\geq t+ 2$ since $a- t- 2c'- 2$ and $c'$ are nonnegative integers, and thus $a_{2}\geq 3$ or $a_{2}= a_{3}= 2, a_{4}= 1$. Color a maximum matching of $H_{2}- D$ by color $1$. If $a_{2}\geq 3$, then color $E_{G}(G[N_{G}[v]\cup D])$ as Figure \ref{f1.2} shows. Let $G_{6}'$ be the graph induced by colored edges. Note that $G'_{6}\not\rightarrow (K_{1, 2}, 3K_{1, 1}, K_{1, 1} \dots, K_{1, 1})$. If $a_{2}= a_{3}= 2$ and $a_{4}= 1$, then color $G[N_{G}[v]\cup D]$ as Figure \ref{f1.3} shows. Let $G_{7}'$ be the graph induced by colored edges. Note that $G'_{7}\not\rightarrow (K_{1, 2}, 2K_{1, 1}, 2K_{1, 1}, K_{1, 1}, \dots, K_{1, 1})$. 
    
    \begin{figure}[H]
      \centering
      \includegraphics[scale=0.5]{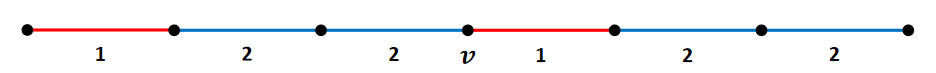}
      \caption{$G[N_{G}(v)\cup D]\not\rightarrow (K_{1, 2}, 3K_{1, 1}, K_{1, 1} \dots, K_{1, 1})$}
      \label{f1.2}
    \end{figure}
    
    \begin{figure}[H]
      \centering
      \includegraphics[scale=0.5]{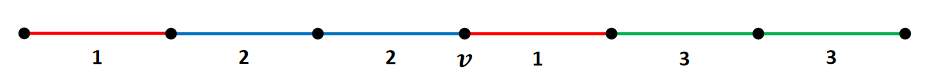}
      \caption{$G[N_{G}(v)\cup D]\not\rightarrow (K_{1, 2}, 2K_{1, 1}, 2K_{1, 1}, K_{1, 1}, \dots, K_{1, 1})$}
      \label{f1.3}
    \end{figure}
    
    In both cases, there are $2c'+ a- t- 2c'- 2= a- t- 2$ uncolored edges. Thus, $G\not\rightarrow \left(a_{1}K_{1, 2}, a_{2}K_{1, 1}, a_{3}K_{1, 1}, \dots, a_{t}K_{1, 1}\right)$ by Fact \ref{fact2.3}. This is a contradiction.\q
  \end{proof}
  
  \begin{theo}\label{theo3.5}
    Suppose that $a_{1}\geq a_{2}\geq 1$ and $a_{3}\geq a_{4}\geq \cdots\geq a_{t}\geq 1$. If $G$ is a size Ramsey minimal graph for $\left(a_{1}K_{1, 2}, a_{2}K_{1, 2}, a_{3}K_{1, 1}, a_{4}K_{1, 1} \dots, a_{t}K_{1, 1}\right)$, then $G= cK_{3}\bigsqcup (a- t+ 1- c)K_{1, 3}$ for some $c\in \{0\}\cup [a- t+ 1]$, where $a= \sum_{s= 1}^{t} a_{s}$. Moreover, if $a_{1}= 1$, then also $G= cK_{3}\bigsqcup c'K_{4}\bigsqcup (a- t+ 1- c- 2c')K_{1, 3}$ for some nonnegative integers $c$ and $c'$ such that $1\leq c+ 2c'\leq a- t+ 1$.
  \end{theo}
  \begin{proof}
    We use induction in $a$. The assertion holds for $a= t$ by Lemma \ref{lemma3.2}. Thus, suppose that the assertion holds for $t- 1$ and $a\geq t+ 1$. Thus, $a_{1}\geq 2$ or $a_{3}\geq 2$. Note that $e(G)= 3(a- t+ 1)$ by Theorem \ref{theo1.5}, and $2\leq \Delta(G)\leq 3$ by Fact \ref{fact3.1} since $b= 3$.
    
    If $\Delta(G)= 2$, then $G$ is the union of paths and cycles. Let $c$ be the number of odd cycles in $G$. Color a maximum matching of $G$ by color $1$. After that, color a maximum matching of the uncolored edges of $G$ by color $2$. Let $G'_{1}$ be the graph induced by colored edges. Note that $G'_{1}\not\rightarrow (K_{1, 2}, K_{1, 2}, K_{1, 1}, \dots, K_{1, 1})$. Moreover, there are $c$ uncolored edges. If $c\leq a- t$, then $G\not\rightarrow \left(a_{1}K_{1, 2}, a_{2}K_{1, 2}, a_{3}K_{1, 1}, a_{4}K_{1, 1}, \dots, a_{t}K_{1, 1}\right)$ by Fact \ref{fact2.3}. This is a contradiction. Thus, $c\geq a- t+ 1$. Note that $3c\leq e(G)= 3(a- t+ 1)$. Consequently, $c= a- t+ 1$. Furthermore, $G= (a- t+ 1)K_{3}$, and we are done. Suppose that $\Delta(G)= 3$. Let $v\in V(G)$ be a vertex such that $d_{G}(v)= \Delta (G)= 3$, and thus $e(G- \{v\})= e(G)- d_{G}(v)= 3(a- t+ 1)- 3= 3(a- t)$. In the following, we divide the discussion into two parts.
    
    {\bf Case 1}. $a_{1}\geq 2$.
    
    Note that $G- \{v\}\rightarrow \left((a_{1}- 1)K_{1, 2}, a_{2}K_{1, 2}, a_{3}K_{1, 1}, a_{4}K_{1, 1}, \dots, a_{t}K_{1, 1}\right)$ by Lemma \ref{lemma2.2} since $a_{1}\geq 2$. Thus, $G- \{v\}$ is a size Ramsey minimal graph by Theorem \ref{theo1.4}. Furthermore, by the induction hypothesis, $G- \{v\}= cK_{3}\bigsqcup (a- t- c)K_{1, 3}$ or $G- \{v\}= cK_{3}\bigsqcup c'K_{4}\bigsqcup (a- t- c- 2c')K_{1, 3}$ if $a_{1}= 2$. Delete all isolated vertices in $G- \{v\}$ and denote the resulting graph by $H_{1}$. If $N_{G}(v)\cap V(H_{1})= \emptyset$, then $G= cK_{3}\bigsqcup (a- t+ 1- c)K_{1, 3}$ or $G= cK_{3}\bigsqcup c'K_{4}\bigsqcup (a- t+ 1- c- 2c')K_{1, 3}$ if $a_{1}= 2$. We only need to consider the latter. Color a copy of $K_{4}$ by color $1$. Moreover, color a maximum matching of the uncolored edges by color $1$ and color the other maximum matching of the uncolored edges by color $2$. Let $G_{2}'$ be the graph induced by colored edges. Note that $G'_{2}\not\rightarrow (2K_{1, 2}, K_{1, 1}, K_{1, 1}, \dots, K_{1, 1})$. Moreover, there are $c+ 2(c'- 1)+ a- t+ 1- c- 2c'= a- t- 1$ uncolored edges. Thus, $G\not\rightarrow \left(a_{1}K_{1, 2}, a_{2}K_{1, 1}, a_{3}K_{1, 1}, \dots, a_{t}K_{1, 1}\right)$ by Fact \ref{fact2.3}. This is a contradiction. Suppose that $N_{G}(v)\cap V(H_{1})\neq\emptyset$. Note that there are at most three components of $H_{1}$ such that each of them contains at least one vertex of $N_{G}(v)$ since $d_{G}(v)= 3$.
    
    {\bf Case 1.1}. There is only one component of $H_{1}$ such that contains at least one vertex of $N_{G}(v)$.
    
    Denote the vertex set of the component by $A$. Note that $H_{1}[A]\neq K_{4}$ since $\Delta(G)= 3$. Thus, $H_{1}- A= (c- i)K_{3}\bigsqcup (a- t- c- j)K_{1, 3}$ or $H_{1}- A= (c- i)K_{3}\bigsqcup c'K_{4}\bigsqcup (a- t- c- 2c'- j)K_{1, 3}$ if $a_{1}= 2$, where $i$ and $j$ are nonnegative integers such that $i+ j= 1$. Moreover, $G[N_{G}[v]\cup A]$ is one of the graphs in Figure \ref{f2.1}.
    
    \begin{figure}[H]
      \centering
      \includegraphics[scale=0.2]{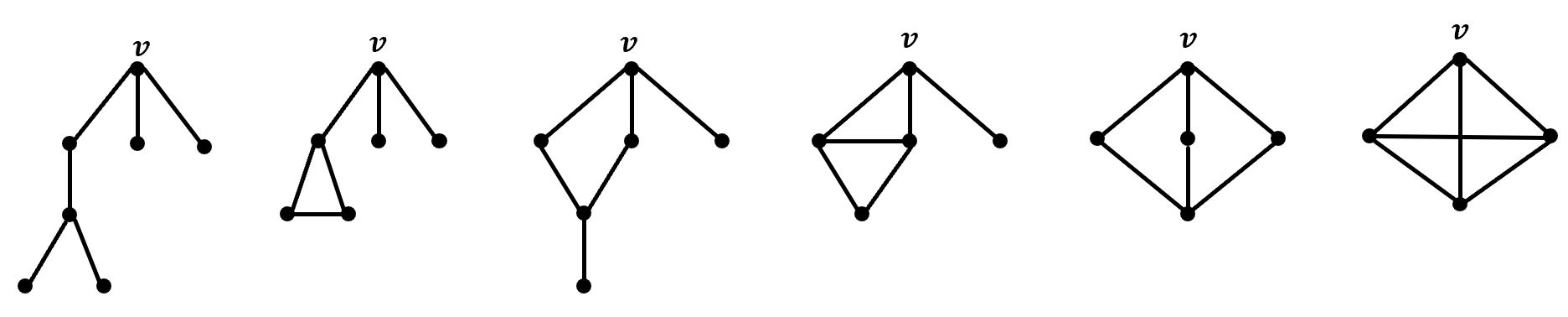}
      \caption{$G[N_{G}[v]\cup A]$}\label{f2.1}
    \end{figure}
    
    We only need to consider the first five cases since the last case is what we need. Color a maximum matching of $H_{1}- A$ by color $1$ and color another maximum matching of $H_{1}- A$ by color $2$. Moreover, color $E_{G}(G[N_{G}(v)\cup A])$ as Figure \ref{f2.2} shows. Let $G_{3}'$ be the graph induced by colored edges. Note that $G'_{3}\not\rightarrow (2K_{1, 2}, K_{1, 2}, K_{1, 1}, K_{1, 1}, \dots, K_{1, 1})$. Moreover, there are $a- t- i- j= a- t- 1$ uncolored edges since $i+ j= 1$. Thus, $G\not\rightarrow \left(a_{1}K_{1, 2}, a_{2}K_{1, 2}, a_{3}K_{1, 1}, a_{4}K_{1, 1}, \dots, a_{t}K_{1, 1}\right)$ by Fact \ref{fact2.3}. This is a contradiction.
    
    \begin{figure}[H]
      \centering
      \includegraphics[scale=0.4]{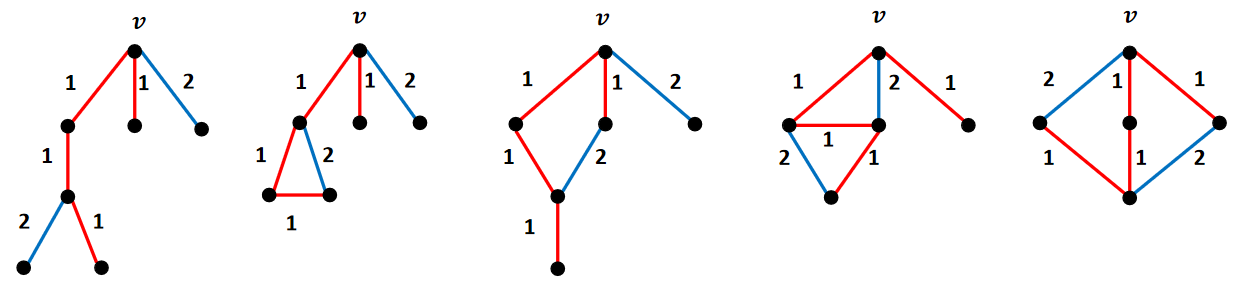}
      \caption{$G[N_{G}[v]\cup A]\not\rightarrow (2K_{1, 2}, K_{1, 2}, K_{1, 1}, K_{1, 1}, \dots, K_{1, 1})$}\label{f2.2}
    \end{figure}
    
    {\bf Case 1.2}. There are two components of $H_{1}$ such that each of them contains at least one vertex of $N_{G}(v)$.
    
    Denote the vertex set of these components by $B_{1}$ and $B_{2}$, respectively. Moreover, let $B= B_{1}\cup B_{2}$. Note that $H_{1}[B_{1}]\neq K_{4}$ and $H_{1}[B_{2}]\neq K_{4}$ since $\Delta(G)= 3$. Thus, $H_{1}- B= (c- i)K_{3}\bigsqcup (a- t- c- j)K_{1, 3}$ or $H_{1}- B= (c- i)K_{3}\bigsqcup c'K_{4}\bigsqcup (a- t- c- 2c'- j)K_{1, 3}$ if $a_{1}= 2$, where $i$ and $j$ are nonnegative integers such that $i+ j= 2$. Moreover, $G[N_{G}[v]\cup B]$ is one of the graphs in Figure \ref{f2.3}.
    
    \begin{figure}[H]
      \centering
      \includegraphics[scale=0.5]{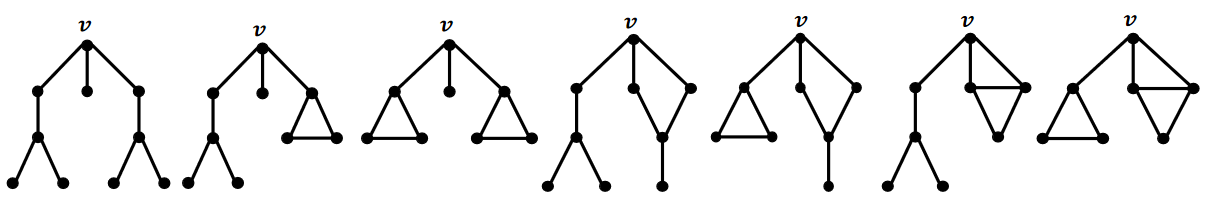}
      \caption{$G[N_{G}[v]\cup B]$}\label{f2.3}
    \end{figure}
    
    Note that $a\geq t+ 2$ since the number of components of $H_{1}- B$ is nonnegative and $i+ j= 2$. Thus, $a_{1}\geq 3$, or $a_{1}= a_{2}= 2$, or $a_{1}= 2, a_{2}= 1$ and $a_{3}\geq 2$. Color a maximum matching of $H_{1}- B$ by color $1$ and color another maximum matching of $H_{1}- B$ by color $2$. If $a_{1}\geq 3$, then color $E_{G}(G[N_{G}(v)\cup B])$ as Figure \ref{f2.4} shows. Let $G_{4}'$ be the graph induced by colored edges. Note that $G'_{4}\not\rightarrow (3K_{1, 2}, K_{1, 2}, K_{1, 1}, K_{1, 1}, \dots, K_{1, 1})$. If $a_{1}= a_{2}= 2$, then color $E_{G}(G[N_{G}(v)\cup B])$ as Figure \ref{f2.5} shows. Let $G_{5}'$ be the graph induced by colored edges. Note that $G'_{5}\not\rightarrow (2K_{1, 2}, 2K_{1, 2}, K_{1, 1}, K_{1, 1}, \dots, K_{1, 1})$. If $a_{1}= 2, a_{2}= 1$ and $a_{3}\geq 2$, then color $E_{G}(G[N_{G}(v)\cup B])$ as Figure \ref{f2.6} shows. Let $G_{6}'$ be the graph induced by colored edges. Note that $G'_{6}\not\rightarrow (2K_{1, 2}, K_{1, 2}, 2K_{1, 1}, K_{1, 1}, \dots, K_{1, 1})$.
    
    \begin{figure}[H]
      \centering
      \includegraphics[scale=0.5]{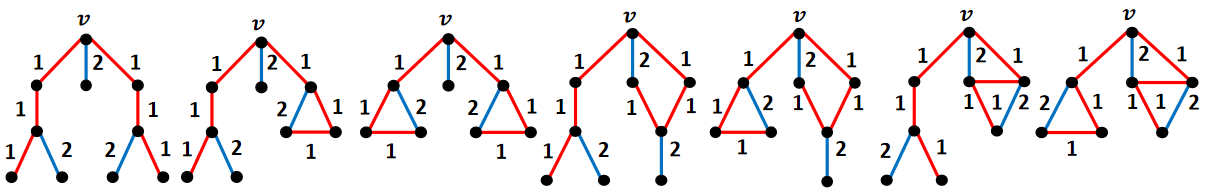}
      \caption{$G[N_{G}(v)\cup B]\not\rightarrow (3K_{1, 2}, K_{1, 2}, K_{1, 1}, K_{1, 1}, \dots, K_{1, 1})$}\label{f2.4}
    \end{figure}
    
    \begin{figure}[H]
      \centering
      \includegraphics[scale=0.5]{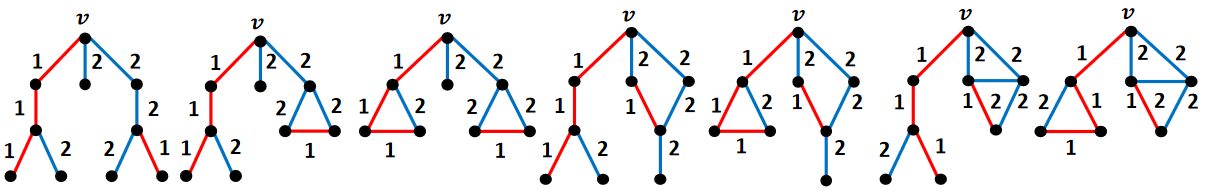}
      \caption{$G[N_{G}(v)\cup B]\not\rightarrow (2K_{1, 2}, 2K_{1, 2}, K_{1, 1}, K_{1, 1}, \dots, K_{1, 1})$}\label{f2.5}
    \end{figure}
    
    \begin{figure}[H]
      \centering
      \includegraphics[scale=0.5]{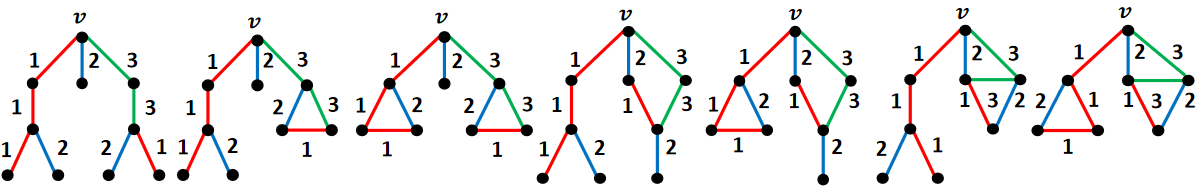}
      \caption{$G[N_{G}(v)\cup B]\not\rightarrow (2K_{1, 2}, K_{1, 2}, 2K_{1, 1}, K_{1, 1}, \dots, K_{1, 1})$}\label{f2.6}
    \end{figure}
    
    In both cases, there are $a- t- i- j= a- t- 2$ uncolored edges since $i+ j= 2$. Thus, $G\not\rightarrow \left(a_{1}K_{1, 2}, a_{2}K_{1, 2}, a_{3}K_{1, 1}, a_{4}K_{1, 1}, \dots, a_{t}K_{1, 1}\right)$ by Fact \ref{fact2.3}. This is a contradiction.
    
    {\bf Case 1.3}. There are three components of $H_{1}$ such that each of them contains at least one vertex of $N_{G}(v)$.
    
    Denote the vertex set of these components by $C_{1}, C_{2}$ and $C_{3}$, respectively. Moreover, let $C= C_{1}\cup C_{2}\cup C_{3}$. Note that $H_{1}[C_{1}]\neq K_{4}, H_{1}[C_{2}]\neq K_{4}$ and $H_{1}[C_{3}]\neq K_{4}$ since $\Delta(G)= 3$. Thus, $H_{1}- C= (c- i)K_{3}\bigsqcup (a- t- c- j)K_{1, 3}$ or $H_{1}- C= (c- i)K_{3}\bigsqcup c'K_{4}\bigsqcup (a- t- c- 2c'- j)K_{1, 3}$ if $a_{1}= 2$, where $i$ and $j$ are nonnegative integers such that $i+ j= 3$. Moreover, $G[N_{G}[v]\cup C]$ is one of the graphs in Figure \ref{f2.7}.
    
    \begin{figure}[H]
      \centering
      \includegraphics[scale=0.4]{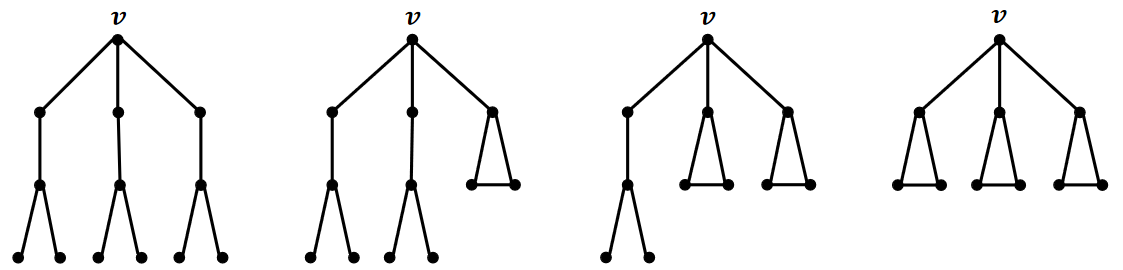}
      \caption{$G[N_{G}[v]\cup C]$}\label{f2.7}
    \end{figure}
    
    Note that $a\geq t+ 3$ since the number of components of $H_{1}- C$ is nonnegative and $i+ j= 3$. Thus, $a_{1}\geq 4$, or $a_{1}=3$ and $a_{2}= 2$, or $a_{1}= 3, a_{2}= 1$ and $a_{3}\geq 2$, or $a_{1}= a_{2}= a_{3}= 2$, or $a_{1}= 2, a_{2}= 1$ and $a_{3}\geq 3$, or $a_{1}= 2, a_{2}= 1$ and $a_{3}= a_{4}= 2$. Color a maximum matching of $H_{1}- B$ by color $1$ and color another maximum matching of $H_{1}- B$ by color $2$. If $a_{1}\geq 4$, then color $E_{G}(G[N_{G}(v)\cup C])$ as Figure \ref{f2.8} shows. Let $G_{7}'$ be the graph induced by colored edges. Note that $G'_{7}\not\rightarrow (4K_{1, 2}, K_{1, 2}, K_{1, 1}, K_{1, 1}, \dots, K_{1, 1})$. If $a_{1}= 3$ and $a_{2}= 2$, then color $E_{G}(G[N_{G}(v)\cup C])$ as Figure \ref{f2.9} shows. Let $G_{8}'$ be the graph induced by colored edges. Note that $G'_{8}\not\rightarrow (3K_{1, 2}, 2K_{1, 2}, K_{1, 1}, K_{1, 1}, \dots, K_{1, 1})$. If $a_{1}= 3, a_{2}= 1$ and $a_{3}\geq 2$, then color $E_{G}(G[N_{G}(v)\cup C])$ as Figure \ref{f2.10} shows. Let $G_{9}'$ be the graph induced by colored edges. Note that $G'_{9}\not\rightarrow (3K_{1, 2}, K_{1, 2}, 2K_{1, 1}, K_{1, 1}, \dots, K_{1, 1})$. If $a_{1}= a_{2}= a_{3}= 2$, then color $E_{G}(G[N_{G}(v)\cup C])$ as Figure \ref{f2.11} shows. Let $G_{10}'$ be the graph induced by colored edges. Note that $G'_{10}\not\rightarrow (2K_{1, 2}, 2K_{1, 2}, 2K_{1, 1}, K_{1, 1}, \dots, K_{1, 1})$. If $a_{1}= 2, a_{2}= 1$ and $a_{3}\geq 3$, then color $E_{G}(G[N_{G}(v)\cup C])$ as Figure \ref{f2.12} shows. Let $G_{11}'$ be the graph induced by colored edges. Note that $G'_{11}\not\rightarrow (2K_{1, 2}, K_{1, 2}, 3K_{1, 1}, K_{1, 1}, \dots, K_{1, 1})$. If $a_{1}= 2, a_{2}= 1$ and $a_{3}= a_{4}= 2$, then color $E_{G}(G[N_{G}(v)\cup C])$ as Figure \ref{f2} shows. Let $G_{12}'$ be the graph induced by colored edges. Note that $G'_{12}\not\rightarrow (2K_{1, 2}, K_{1, 2}, 2K_{1, 1}, 2K_{1, 1}, K_{1, 1}, \dots, K_{1, 1})$.
    
    \begin{figure}[H]
      \centering
      \includegraphics[scale=0.4]{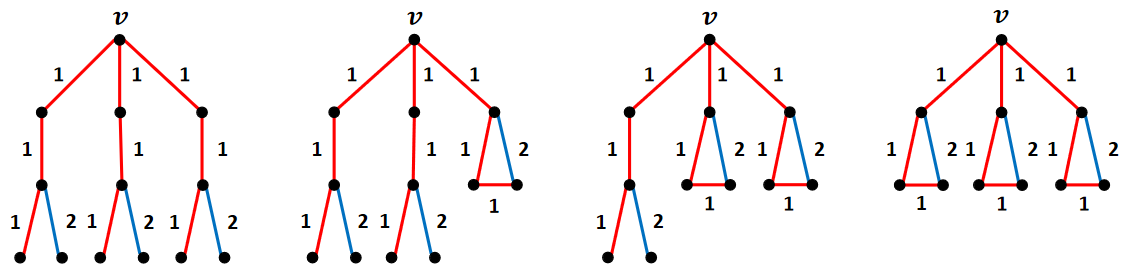}
      \caption{$G[N_{G}(v)\cup C]\not\rightarrow (4K_{1, 2}, K_{1, 2}, K_{1, 1}, K_{1, 1}, \dots, K_{1, 1})$}\label{f2.8}
    \end{figure}
    
    \begin{figure}[H]
      \centering
      \includegraphics[scale=0.4]{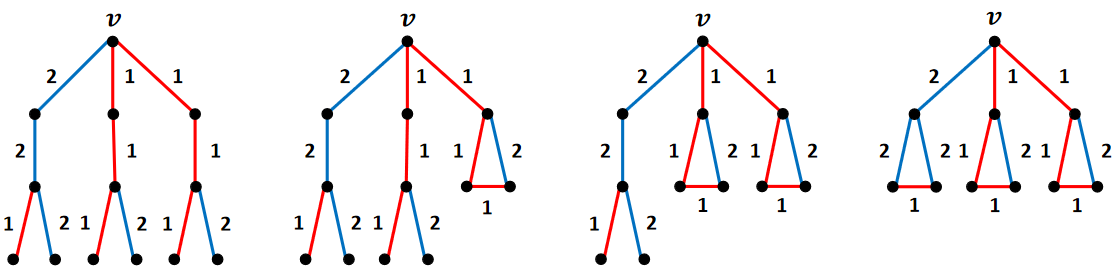}
      \caption{$G[N_{G}(v)\cup C]\not\rightarrow (3K_{1, 2}, 2K_{1, 2}, K_{1, 1}, K_{1, 1}, \dots, K_{1, 1})$}\label{f2.9}
    \end{figure}
    
    \begin{figure}[H]
      \centering
      \includegraphics[scale=0.4]{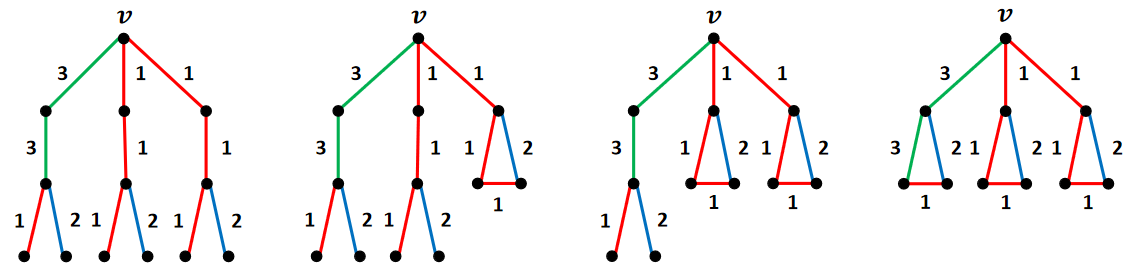}
      \caption{$G[N_{G}(v)\cup C]\not\rightarrow (3K_{1, 2}, K_{1, 2}, 2K_{1, 1}, K_{1, 1}, \dots, K_{1, 1})$}\label{f2.10}
    \end{figure}
    
    \begin{figure}[H]
      \centering
      \includegraphics[scale=0.4]{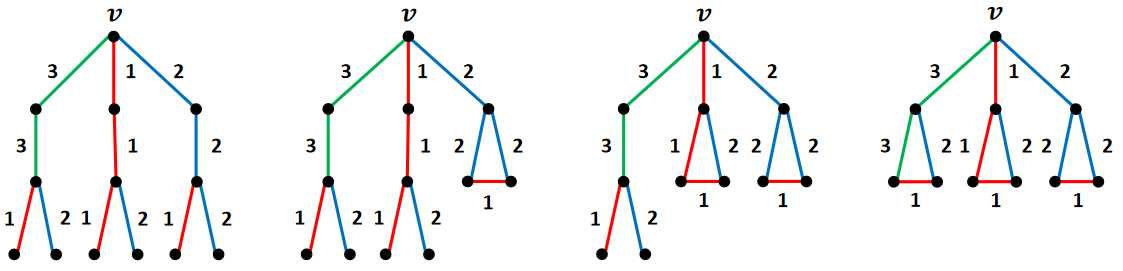}
      \caption{$G[N_{G}(v)\cup C]\not\rightarrow (2K_{1, 2}, 2K_{1, 2}, 2K_{1, 1}, K_{1, 1}, \dots, K_{1, 1})$}\label{f2.11}
    \end{figure}
    
    \begin{figure}[H]
      \centering
      \includegraphics[scale=0.4]{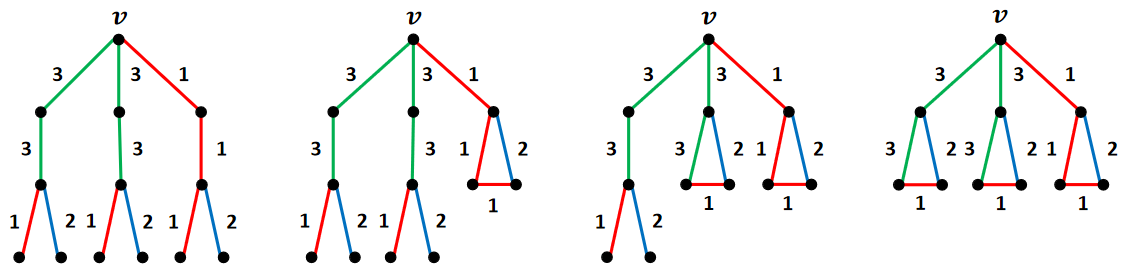}
      \caption{$G[N_{G}(v)\cup C]\not\rightarrow (2K_{1, 2}, K_{1, 2}, 3K_{1, 1}, K_{1, 1}, \dots, K_{1, 1})$}\label{f2.12}
    \end{figure}
    
    \begin{figure}[H]
      \centering
      \includegraphics[scale=0.4]{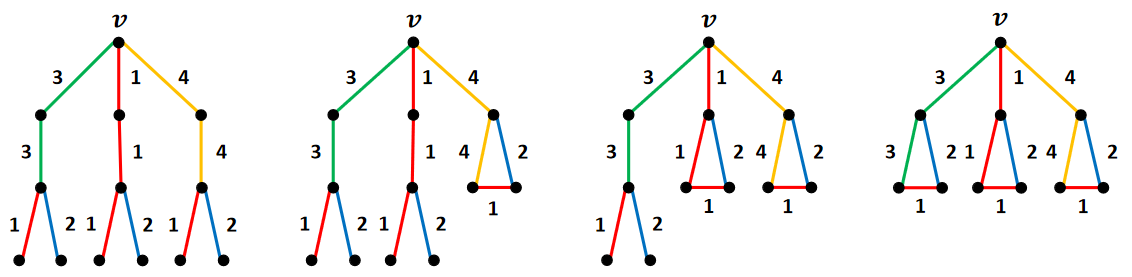}
      \caption{$G[N_{G}(v)\cup C]\not\rightarrow (2K_{1, 2}, K_{1, 2}, 2K_{1, 1}, 2K_{1, 1}, K_{1, 1}, \dots, K_{1, 1})$}\label{f2}
    \end{figure}
    
    In both cases, there are $a- t- i- j= a- t- 3$ uncolored edges since $i+ j= 3$. Thus, $G\not\rightarrow \left(a_{1}K_{1, 2}, a_{2}K_{1, 2}, a_{3}K_{1, 1}, a_{4}K_{1, 1}, \dots, a_{t}K_{1, 1}\right)$ by Fact \ref{fact2.3}. This is a contradiction.
    
    {\bf Case 2}. $a_{1}= 1$.
    
    In this case, we only need to show that $G= cK_{3}\bigsqcup c'K_{4}\bigsqcup (a- t+ 1- c- 2c')K_{1, 3}$ for some nonnegative integers $c$ and $c'$ such that $1\leq c+ 2c'\leq a- t+ 1$. Note that $a_{3}\geq 2$ by our assumption, and thus $G- \{v\}\rightarrow \left(K_{1, 2}, K_{1, 2}, (a_{3}- 1)K_{1, 1}, a_{4}K_{1, 1}, \dots, a_{t}K_{1, 1}\right)$ by Lemma \ref{lemma2.2} since $a_{3}\geq 2$. Thus, $G- \{v\}$ is a size Ramsey minimal graph by Theorem \ref{theo1.4}. Moreover, by the induction hypothesis, $G- \{v\}= cK_{3}\bigsqcup c'K_{4}\bigsqcup (a- t- c- 2c')K_{1, 3}$. Delete all isolated vertices in $G- \{v\}$ and denote the resulting graph by $H_{2}$. If $N_{G}(v)\cap V(H_{2})= \emptyset$, then $G= cK_{3}\bigsqcup c'K_{4}\bigsqcup (a- t+ 1- c- 2c')K_{1, 3}$, and we are done. Suppose that $N_{G}(v)\cap V(H_{2})\neq\emptyset$. Note that there are at most three components of $H_{2}$ such that each of them contains at least one vertex of $N_{G}(v)$ since $d_{G}(v)= 3$.
    
    {\bf Case 2.1}. There is only one component of $H_{2}$ such that contains at least one vertex of $N_{G}(v)$.
    
    Denote the vertex set of the component by $A'$. Note that $H_{2}[A']\neq K_{4}$ since $\Delta(G)= 3$. Thus, $H_{2}- A'= (c- i)K_{3}\bigsqcup c'K_{4}\bigsqcup (a- t- c- 2c'- j)K_{1, 3}$, where $i$ and $j$ are nonnegative integers such that $i+ j= 1$. Moreover, $G[N_{G}[v]\cup A']$ is one of the graphs in Figure \ref{f2.1}. We only need to consider the first five cases since the last case is what we need. Color a maximum matching of $H_{2}- A'$ by color $1$ and color another maximum matching of $H_{2}- A'$ by color $2$. Moreover, color $E_{G}(G[N_{G}(v)\cup A'])$ as Figure \ref{f2.13} shows. Let $G_{13}'$ be the graph induced by colored edges. Note that $G'_{13}\not\rightarrow (K_{1, 2}, K_{1, 2}, 2K_{1, 1}, K_{1, 1}, \dots, K_{1, 1})$. Moreover, there are $c- i+ 2c'+ a- t- c- 2c'- j= a- t- 1$ uncolored edges since $i+ j= 1$. Thus, $G\not\rightarrow \left(K_{1, 2}, K_{1, 2}, a_{3}K_{1, 1}, a_{4}K_{1, 1}, \dots, a_{t}K_{1, 1}\right)$ by Fact \ref{fact2.3}. This is a contradiction.
    
    \begin{figure}[H]
      \centering
      \includegraphics[scale=0.4]{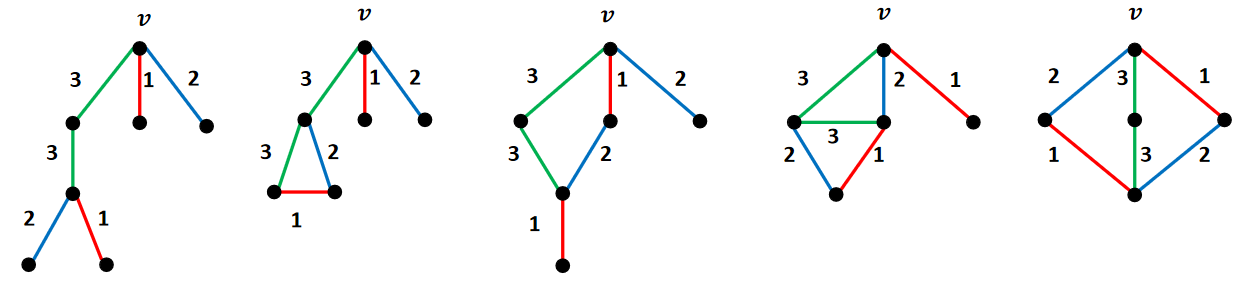}
      \caption{$G[N_{G}[v]\cup A']\not\rightarrow (K_{1, 2}, K_{1, 2}, 2K_{1, 1}, K_{1, 1}, \dots, K_{1, 1})$}\label{f2.13}
    \end{figure}
    
    {\bf Case 2.2}. There are two components of $H_{2}$ such that each of them contains at least one vertex of $N_{G}(v)$.
    
    Denote the vertex set of these components by $B_{1}'$ and $B_{2}'$, respectively. Moreover, let $B'= B_{1}'\cup B_{2}'$. Note that $H_{2}[B_{1}']\neq K_{4}$ and $H_{2}[B_{2}']\neq K_{4}$ since $\Delta(G)= 3$. Thus, $H_{2}- B'= (c- i)K_{3}\bigsqcup c'K_{4}\bigsqcup (a- t- c- 2c'- j)K_{1, 3}$, where $i$ and $j$ are nonnegative integers such that $i+ j= 2$. Moreover, $G[N_{G}[v]\cup B']$ is one of the graphs in Figure \ref{f2.3}. Note that $a\geq t+ 2$ since the number of components of $H_{2}- B'$ is nonnegative and $i+ j= 2$. Thus, $a_{3}\geq 3$, or $a_{3}= a_{4}= 2$. Color a maximum matching of $H_{2}- B$ by color $1$ and color another maximum matching of $H_{2}- B$ by color $2$. If $a_{3}\geq 3$, then color $E_{G}(G[N_{G}(v)\cup B'])$ as Figure \ref{f2.14} shows. Let $G_{14}'$ be the graph induced by colored edges. Note that $G'_{14}\not\rightarrow (K_{1, 2}, K_{1, 2}, 3K_{1, 1}, K_{1, 1}, \dots, K_{1, 1})$. If $a_{3}= a_{4}= 2$, then color $E_{G}(G[N_{G}(v)\cup B'])$ as Figure \ref{f2.15} shows. Let $G_{15}'$ be the graph induced by colored edges. Note that $G'_{15}\not\rightarrow (K_{1, 2}, K_{1, 2}, 2K_{1, 1}, 2K_{1, 1}, K_{1, 1}, \dots, K_{1, 1})$.
    
    \begin{figure}[H]
      \centering
      \includegraphics[scale=0.5]{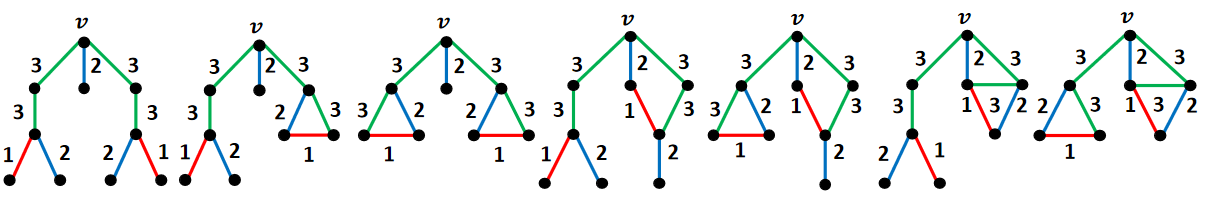}
      \caption{$G[N_{G}(v)\cup B']\not\rightarrow (K_{1, 2}, K_{1, 2}, 3K_{1, 1}, K_{1, 1}, \dots, K_{1, 1})$}\label{f2.14}
    \end{figure}
    
    \begin{figure}[H]
      \centering
      \includegraphics[scale=0.5]{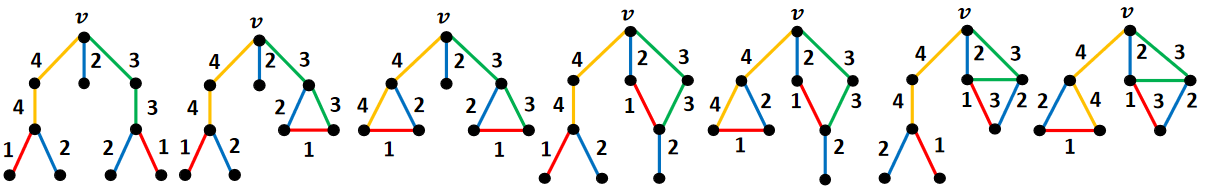}
      \caption{$G[N_{G}(v)\cup B']\not\rightarrow (K_{1, 2}, K_{1, 2}, 2K_{1, 1}, 2K_{1, 1}, K_{1, 1}, \dots, K_{1, 1})$}\label{f2.15}
    \end{figure}
    
    In both cases, there are $c- i+ 2c'+ a- t- c- 2c'- j= a- t- 2$ uncolored edges since $i+ j= 2$. Thus, $G\not\rightarrow \left(K_{1, 2}, K_{1, 2}, a_{3}K_{1, 1}, a_{4}K_{1, 1}, \dots, a_{t}K_{1, 1}\right)$ by Fact \ref{fact2.3}. This is a contradiction.
    
    {\bf Case 2.3}. There are three components of $H_{2}$ such that each of them contains at least one vertex of $N_{G}(v)$.
    
    Denote the vertex set of these components by $C_{1}', C_{2}'$ and $C_{3}'$, respectively. Moreover, let $C'= C_{1}'\cup C_{2}'\cup C_{3}'$. Note that $H_{2}[C_{1}']\neq K_{4}, H_{2}[C_{2}']\neq K_{4}$ and $H_{2}[C_{3}']\neq K_{4}$ since $\Delta(G)= 3$. Thus, $H_{2}- C'= (c- i)K_{3}\bigsqcup c'K_{4}\bigsqcup (a- t- c- 2c'- j)K_{1, 3}$, where $i$ and $j$ are nonnegative integers such that $i+ j= 3$. Moreover, $G[N_{G}[v]\cup C']$ is one of the graphs in Figure \ref{f2.7}. Note that $a\geq t+ 3$ since the number of components of $H_{2}- C'$ is nonnegative and $i+ j= 3$. Thus, $a_{3}\geq 4$, or $a_{3}= 3$ and $a_{4}= 2$, or $a_{3}= a_{4}= a_{5}= 2$. Color a maximum matching of $H_{2}- C'$ by color $1$ and color another maximum matching of $H_{2}- C'$ by color $2$. If $a_{3}\geq 4$, then color $E_{G}(G[N_{G}(v)\cup C'])$ as Figure \ref{f2.16} shows. Let $G_{16}'$ be the graph induced by colored edges. Note that $G'_{16}\not\rightarrow (K_{1, 2}, K_{1, 2}, 4K_{1, 1}, K_{1, 1}, \dots, K_{1, 1})$. If $a_{3}= 3$ and $a_{4}= 2$, then color $E_{G}(G[N_{G}(v)\cup C'])$ as Figure \ref{f2.17} shows. Let $G_{17}'$ be the graph induced by colored edges. Note that $G'_{17}\not\rightarrow (K_{1, 2}, K_{1, 2}, 3K_{1, 1}, 2K_{1, 1}, K_{1, 1}, \dots, K_{1, 1})$. If $a_{3}= a_{4}= a_{5}= 2$, then color $E_{G}(G[N_{G}(v)\cup C'])$ as Figure \ref{f2.18} shows. Let $G_{18}'$ be the graph induced by colored edges. Note that $G'_{18}\not\rightarrow (K_{1, 2}, K_{1, 2}, 2K_{1, 1}, 2K_{1, 1}, 2K_{1, 1}, K_{1, 1}, \dots, K_{1, 1})$.
    
    \begin{figure}[H]
      \centering
      \includegraphics[scale=0.4]{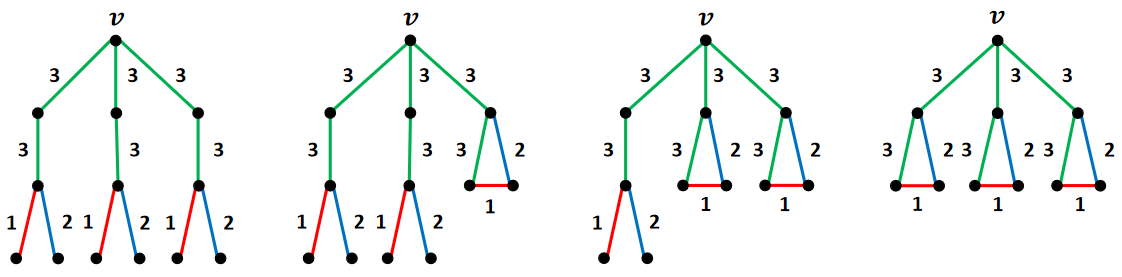}
      \caption{$G[N_{G}(v)\cup C]\not\rightarrow (K_{1, 2}, K_{1, 2}, 4K_{1, 1}, K_{1, 1}, \dots, K_{1, 1})$}\label{f2.16}
    \end{figure}
    
    \begin{figure}[H]
      \centering
      \includegraphics[scale=0.19]{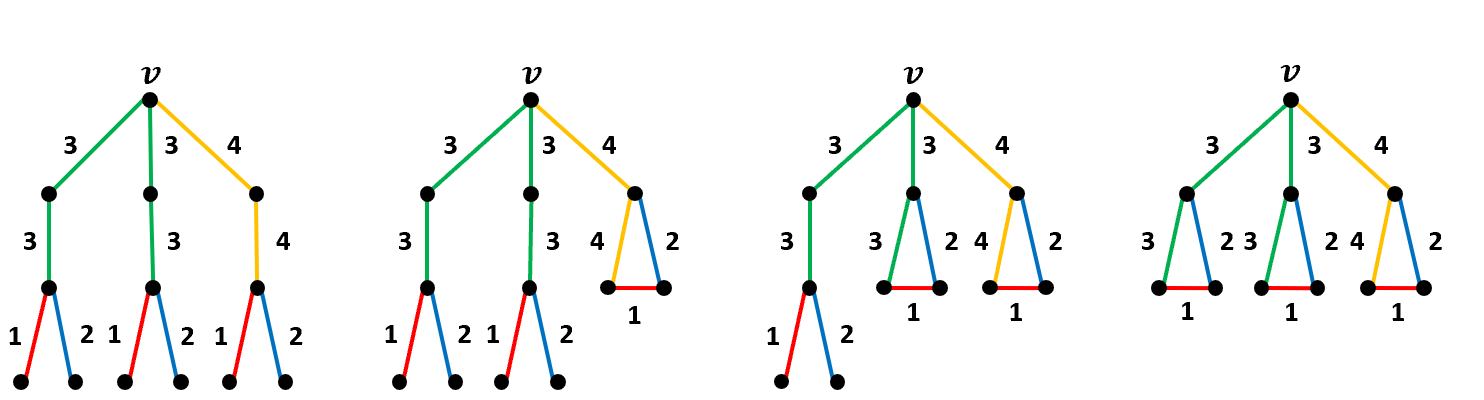}
      \caption{$G[N_{G}(v)\cup C]\not\rightarrow (K_{1, 2}, K_{1, 2}, 3K_{1, 1}, 2K_{1, 1}, K_{1, 1}, \dots, K_{1, 1})$}\label{f2.17}
    \end{figure}
    
    \begin{figure}[H]
      \centering
      \includegraphics[scale=0.4]{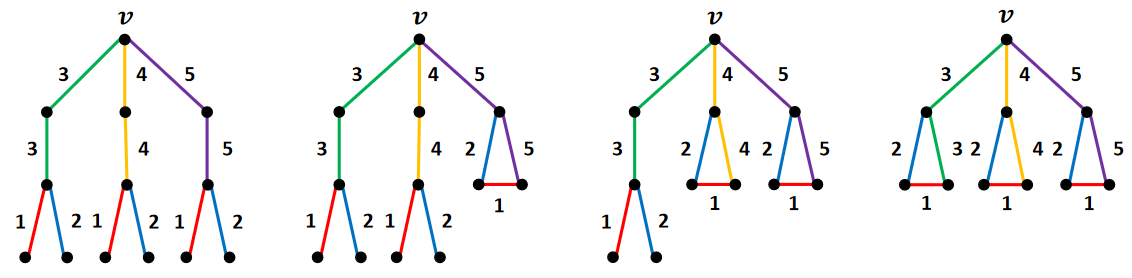}
      \caption{$G[N_{G}(v)\cup C]\not\rightarrow (K_{1, 2}, K_{1, 2}, 2K_{1, 1}, 2K_{1, 1}, 2K_{1, 1}, K_{1, 1}, \dots, K_{1, 1})$}\label{f2.18}
    \end{figure}
    
    In both cases, there are $c- i+ 2c'+ a- t- c- 2c'- j= a- t- 3$ uncolored edges since $i+ j= 3$. Thus, $G\not\rightarrow \left(K_{1, 2}, K_{1, 2}, a_{3}K_{1, 1}, a_{4}K_{1, 1}, \dots, a_{t}K_{1, 1}\right)$ by Fact \ref{fact2.3}. This is a contradiction.\q
  \end{proof}
  
  \begin{theo}\label{theo3.6}
    Suppose that $b_{1}\geq b_{2}\geq \cdots\geq b_{t}\geq 1$. Moreover, suppose that $b_{1}\geq 3$, or $b_{1}= b_{2}= \cdots= b_{\ell}= 2$ and $b_{\ell+ 1}= b_{\ell+ 2}= \cdots= b_{t}= 1$ for some $\ell\in [t]\backslash [2]$. If $G$ is a size Ramsey minimal graph for $(a_{1}K_{1, b_{1}}, a_{2}K_{1, b_{2}}, \dots, a_{t}K_{1, b_{t}})$, then $G= (a- t+ 1)K_{1, b}$, where $a= \sum_{s= 1}^{t} a_{s}$ and $b= \sum_{s= 1}^{t} b_{s}- t+ 1$.
  \end{theo}
  \begin{proof}  
    Note that $b\geq 3$ by our assumption. We use induction in $a$. The assertion holds for $a= t$ by Lemma \ref{lemma3.2}. Suppose that the assertion holds for $a- 1$ and $a\geq t+ 1$. Thus, there is an $i\in [t]$ such that $a_{i}\geq 2$. Moreover, $e(G)= (a- t+ 1)b$ by Theorem \ref{theo1.4}, and $b- 1\leq \Delta(G)\leq b$ by Fact \ref{fact3.1}. Let $v_{1}\in V(G)$ be a vertex such that $d_{G}(v_{1})= \Delta (G)$. In the following, we divide our discussion into two parts.
    
    {\bf Case 1}. $d_{G}(v_{1})= b$.
    
    Note that $e(G- \{v_{1}\})= e(G)- d(v_{1})= (a- t+ 1)b- b= (a- t)b$. Moreover, $G\rightarrow (a_{1}K_{1, b_{1}}, \dots, a_{i- 1}K_{1, b_{i- 1}}, (a_{i}- 1)K_{1, b_{i}}, a_{i+ 1}K_{1, b_{i+ 1}}, \dots, a_{t}K_{1, b_{t}})$ by Lemma \ref{lemma2.2} since $a_{i}\geq 2$. Thus, $G- \{v_{1}\}$ is a size Ramsey minimal graph by Theorem \ref{theo1.4}. Furthermore, $G- \{v_{1}\}= (a- t)K_{1, b}$ by the induction hypothesis. Delete all isolated vertices in $G- \{v_{1}\}$ and denote the resulting graph by $H$. If $N_{G}(v_{1})\cap V(H)= \emptyset$, then $G= (a- t+ 1)K_{1, b}$, and we are done. 
    
    Suppose that $N_{G}(v)\cap V(H)\neq\emptyset$. Note that the center of each copy of $K_{1, b}$ of $H$ is not belong to $N_{G}(v)$ since $\Delta(G)= b$. Let $A$ be the vertex set of $K_{1, b}$, which contains at least one vertex of $N_{G}(v)$. Moreover, denote one of the common vertices by $u$. Color all edges incident to $u$ by color $i$. Note that the uncolored edges of $G[N_{G}[v]\cup A]$ are the union of $b- 1$ matchings since $b\geq 3$. Color $b_{s}- 1$ matchings of the uncolored edges of $G[N_{G}[v]\cup A]$ by color $s$ for each $s\in [t]$. Let $G'_{1}$ be the graph induced by colored edges, and thus $G'_{1}\not\rightarrow (K_{1, b_{1}}, \dots, K_{1, b_{i- 1}}, 2K_{1, b_{i}}, K_{1, b_{i+ 1}}, \dots, K_{1, b_{t}})$. Note that $H- A= (a- t- 1)K_{1, b}$ and the centers not belong to $V(G_{1}')$ since $\Delta(G)= b$. Thus, $G\not\rightarrow (a_{1}K_{1, b_{1}}, a_{2}K_{1, b_{2}}, \dots, a_{t}K_{1, b_{t}})$ by Fact \ref{fact2.4}. This is a contradiction.
    
    {\bf Case 2}. $d_{G}(v_{1})= b- 1$.
    
    Let $G_{1}= G- \{v_{1}\}$. Then
    $$G_{1}\rightarrow (a_{1}K_{1, b_{1}}, \dots, a_{i- 1}K_{1, b_{i- 1}}, (a_{i}- 1)K_{1, b_{i}}, a_{i+ 1}K_{1, b_{i+ 1}}, \dots, a_{t}K_{1, b_{t}})$$
    by Lemma \ref{lemma2.2} since $a_{i}\geq 2$. If $\Delta(G_{1})\leq b- 2$, then $G_{1}\not\rightarrow (K_{1, b_{1}}, K_{1, b_{2}}, \dots, K_{1, b_{t}}) $ by Fact \ref{fact2.1}. This is a contradiction. Thus, $\Delta(G_{1})= b- 1$ since $\Delta(G)= d_{G}(v_{1})= b- 1$. Let $v_{2}\in V(G_{1})$ be a vertex such that $d_{G_{1}}(v_{2})= \Delta(G_{1})= b- 1$. Note that $v_{1}$ and $v_{2}$ are not adjacent since $\Delta(G)= b- 1$, and thus $d_{G}(v_{2})= d_{G_{1}}(v_{2})$. Repeating the process, we can select $a- t+ 1$ vertices ($v_{1}, v_{2}, \dots, v_{a- t+ 1}$ and denote them by $U$) step by step such that $G[U]= \emptyset$ and $d_{G}(v_{i})= b- 1$ for each $i\in [a- t+ 1]$. Let $W= V(G)- U$. Note that $G[U, W]$ is a bipartite graph, and then $\chi'(G[U, W])= \Delta(G[U, W])= b- 1$ by the K{\"o}nig Theorem \cite{konig}. Thus, there is a proper edge coloring of $G[U, W]$ with $b- 1$ colors (denote the colors by $1', 2', \dots, (b- 1)'$), that is, the graph induced by the edges in color $i'$ is a matching for each $i\in [b- 1]$. Let $b_{0}= 0$. For each $s\in [t]$, recolor the edges in colors of the color set $\left\{\sum_{i= 1}^{s} b_{i- 1}- s+ 2, \sum_{i= 1}^{s} b_{i- 1}- s+ 3, \dots, \sum_{i= 1}^{s} b_{i}- s\right\}$ by color $s$. Thus, there is no monochromatic copy of $K_{1, b_{s}}$ for each $s\in [t]$ in $G[U, W]$. Note that 
    $$e(G[W])= e(G)- e(G[U, W])= (a- t+ 1)b- (a- t+ 1)(b- 1)= a- t+ 1$$
    since $G[U]= \emptyset$. If at least two edges in $G[W]$ are incident, then color two of them by color $i$. Let $G'_{2}$ be the graph induced by colored edges. Note that 
    $$G'_{2}\not\rightarrow (K_{1, b_{1}}, \dots, K_{1, b_{i- 1}}, 2K_{1, b_{i}}, K_{1, b_{i+ 1}}, \dots, K_{1, b_{t}}).$$
    Moreover, there are $a- t+ 1- 2= a- t- 1$ uncolored edges. Thus, 
    $$G\not\rightarrow (a_{1}K_{1, b_{1}}, a_{2}K_{1, b_{2}}, \dots, a_{t}K_{1, b_{t}})$$
    by Fact \ref{fact2.3}. This is a contradiction.
    
    Thus, $G[W]$ is a matching and denote the edges by $\{u_{i}w_{i}: i\in [a- t+ 1]\}$. If there is an $i_{0}\in [a- t+ 1]$ such that at most $b- 2$ edges of $G[U, W]$ are incident to $u_{i_{0}}$ and $w_{i_{0}}$, then there is $s_{0}\in [t]$ such that at most $b_{s_{0}}- 2$ edges of $G[U, W]$ are incident to $u_{i}$ and $w_{i}$. Color $u_{i}w_{i}$ by color $s_{0}$. Let $G_{3}'$ be the graph induced by colored edges. Note that $G'_{3}\not\rightarrow (K_{1, b_{1}}, K_{1, b_{2}}, \dots, K_{1, b_{t}})$. Moreover, there are $e(G[W])- 1= a- t$ uncolored edges. Thus, $G\not\rightarrow (a_{1}K_{1, b_{1}}, a_{2}K_{1, b_{2}}, \dots, a_{t}K_{1, b_{t}})$ by Fact \ref{fact2.3}. This is a contradiction. Consequently, for each $i\in [a- t+ 1]$, $b- 1$ edges of $G[U, W]$ are incident to $u_{i}$ and $w_{i}$ since $e(G[U, W])= (b- 1)(a- t+ 1)$. Furthermore, at least one edge of $G[U, W]$ is incident to $u_{i}$ and at least one edge of $G[U, W]$ is incident to $w_{i}$ since $\Delta(G)= b- 1$.
    
    {\bf Subcase 2.1}. $b_{1}\geq 3$. 
    
    In fact, there is an $i_{1}\in [a- t+ 1]$ such that at most $b_{1}- 2$ edges of $G[U, W]$ in color $1$ are incident to $u_{i_{1}}$ and at most $b_{1}- 2$ edges of $G[U, W]$ in color $1$ are incident to $w_{i_{1}}$. Otherwise, for each $i\in [a- t+ 1]$, at least $b_{1}- 1$ edges of $G[U, W]$ in color $1$ are incident to $u_{i}$ or at least $b_{1}- 1$ edges of $G[U, W]$ in color $1$ are incident to $w_{i}$. Note that there are exactly $(a- t+ 1)(b_{1}- 1)$ edges of $G[U, W]$ in color $1$. Thus, no edge of $G[U, W]$ in color $1$ is incident to one of $u_{i}$ and $w_{i}$ for each $i\in [a- t+ 1]$. Without loss of generality, suppose that no edge of $G[U, W]$ in color $1$ is incident to $u_{1}$, and thus $b_{1}- 1$ edges of $G[U, W]$ in color $1$ are incident to $w_{1}$. Recall that at least one edge of $G[U, W]$ is incident to $u_{1}$. Note that in the original coloring, the edge is in color $j'$ for some $j\in [b- 1]\backslash [b_{1}- 1]$. Without loss of generality, suppose that $j= b- 1$. In the color set $\{1', 2', \dots, (b- 1)'\}$, recolor the edges in color $1'$ by color $t$ (which is colored by color $1$) and recolor the edges in color $(b- 1)'$ by color $1$ (which is colored by color $t$). After the new recoloring, in $G[U, W]$, $b_{1}- 2$ edges in color $1$ are incident to $u_{1}$ and one edge in color $1$ is incident to $w_{1}$. Moreover, there is still no monochromatic copy of $K_{1, b_{s}}$ for each $s\in [t]$ in $G[U, W]$. We finish the proof of the fact.
    
    Color $u_{i_{1}}w_{i_{1}}$ by color $1$. Let $G'_{4}$ be the graph induced by colored edges. Note that $G'_{4}\not\rightarrow (K_{1, b_{1}}, K_{1, b_{2}}, \dots, K_{1, b_{t}})$ since $b_{1}\geq 3$. Moreover, there are $e(G[W])- 1= a- t$ uncolored edges. Thus, $G\not\rightarrow (a_{1}K_{1, b_{1}}, a_{2}K_{1, b_{2}}, \dots, a_{t}K_{1, b_{t}})$ by Fact \ref{fact2.3}. This is a contradiction.
    
    {\bf Subcase 2.2}. $b_{1}= b_{2}= \cdots= b_{\ell}= 2$ and $b_{\ell+ 1}= b_{\ell+ 2}= \cdots= b_{t}= 1$ for some $\ell\in [t]\backslash [2]$.
    
    Suppose that there is $s_{1}\in [\ell]$ such the no edge of $G[U, W]$ in color $s_{1}$ is incident to $u_{1}$ and no edge of $G[U, W]$ in color $s_{1}$ is incident to $w_{1}$. Color $u_{1}w_{1}$ by color $s_{1}$. Let $G'_{5}$ be the graph induced by colored edges. Note that $G'_{5}\not\rightarrow (K_{1, b_{1}}, K_{1, b_{2}}, \dots, K_{1, b_{t}})$. Moreover, there are $e(G[W])- 1= a- t$ uncolored edges. Thus, $G\not\rightarrow (a_{1}K_{1, b_{1}}, a_{2}K_{1, b_{2}}, \dots, a_{t}K_{1, b_{t}})$ by Fact \ref{fact2.3}. This is a contradiction.
    
    Thus, for each $s\in [\ell]$, there is exactly one edge of $G[U, W]$ in color $s$ is incident to $u_{1}$ or $w_{1}$, since there are $b- 1= \ell$ edges of $G[U, W]$ are incident to $u_{1}$ and $w_{1}$, and there is no monochromatic copy of $K_{1, 2}$ in color $s$ for each $s\in [\ell]$ in $G[U, W]$. For convenience in the following, suppose that $d_{G}(u_{1})\geq d_{G}(w_{1})$. By our assumption, there is at least one color not appear (denote one of them by $k$) in those edges of $G[U, W]$ are incident to $u_{1}$ since at least one edge of $G[U, W]$ is incident to $w_{1}$. Moreover, at least two colors (denote two of them by $k'$ and $k''$) not appear in those edges of $G[U, W]$ are incident to $w_{1}$ since $d_{G}(u_{1})\geq d_{G}(w_{1})$ and $\ell\in [t]\backslash [2]$.
    
    Let $u_{1}w_{1}v'v''$ be a copy of $P_{4}$ in $G$. Note that $v'\in U$ and $v''\in W$ since $u_{1}, w_{1}\in U$, $G[U]= \emptyset$ and $G[W]$ is a matching. Moreover, we can select the vertices $v'$ and $v''$ such that $w_{1}v'$ is in color $k$ and $v'v''$ is in color $k'$ or $k''$ by our assumption. Without loss of generality, suppose that $v'v''$ is in color $k'$. Note that the graph induced by edges in color $s$ is a matching for each $s\in [\ell]$ by our coloring. We recolor $w_{1}v'$ by color $k'$ and $v'v''$ by color $k$. Thus, there is still no monochromatic copy of $K_{1, b_{s}}$ for each $s\in [t]$ in $G[U, W]$. Color $u_{1}w_{1}$ by color $k$. Let $G'_{6}$ be the graph induced by colored edges. Note that $G'_{6}\not\rightarrow (K_{1, b_{1}}, K_{1, b_{2}}, \dots, K_{1, b_{t}})$ since no edge of $G[U, W]$ in color $k$ is incident to $u_{1}$ and $w_{1}$. Moreover, there are $e(G[W])- 1= a- t$ uncolored edges. Thus, $G\not\rightarrow (a_{1}K_{1, b_{1}}, a_{2}K_{1, b_{2}}, \dots, a_{t}K_{1, b_{t}})$ by Fact \ref{fact2.3}. This is a contradiction.\q
  \end{proof}

\section*{Acknowledgments}
  The authors thank Akbar Davoodi for some helpful suggestions. The authors also thank Liying Kang and Yuejian Peng.
  
\section*{Conflict of interest}
The authors declare that they have no conflict of interest.

\section*{Data availability}
No data was used for the research described in the paper.


\begin{thebibliography}{JluR00}
        
    \bibitem{bls} C. Beke, A. Li, and J. Sahasrabudhe, The multicolour size Ramsey number of a path, arXiv:2511.16656.
        
    \bibitem{BurrErdosFaudreeRousseauSchelp} S. A. Burr, P. Erd\H{o}s, R. J. Faudree, C. C. Rousseau and R. H. Schelp, Ramsey-minimal graphs for multiple copies, \emph{Nederl. Akad. Wetensch. Indag. Math.}, \textbf{81} (1978), 187--195.
        
    \bibitem{Davoodi} A. Davoodi, R. Javadi, A. Kamranian and G. Raeisi, On a conjecture of Erd\H{o}s on size Ramsey number of star forests, \emph{Ars Math. Contemp.}, \textbf{25} (2025), \#P2.09.
        
    \bibitem{nk} N. Dragani{\'c} and K. Petrova, Size-Ramsey numbers of graphs with maximum degree three, \emph{J. Lond. Math. Soc.}, \textbf{111} (2025), e70116.
        
    \bibitem{ErdosFaudreeRousseauSchelp} P. Erd\H{o}s, R. J. Faudree, C. C. Rousseau and R. H. Schelp, The size Ramsey number, \emph{Period. Math. Hungar.}, \textbf{9} (1978), 145--161.
        
    \bibitem{FaudreeSchelp} R. J. Faudree and R. H. Schelp, A survey of results on the size Ramsey number, in: {\it Paul Erd\H{o}s and his mathematics, II (Budapest, 1999)}, J\'anos Bolyai Math. Soc., Budapest, volume 11 of {\it Bolyai Soc. Math. Stud.}, pp. 291--309, 2002.
        
    \bibitem{GyoriSchelp} E. Gy\H{o}ri and R. H. Schelp, Two-edge colorings of graphs with bounded degree in both colors, \emph{Discrete Math.}, \textbf{249} (2002), 105--110.
        
    \bibitem{JavadiOmidi} R. Javadi and G. Omidi, On a question of Erd\H{o}s and Faudree on the size Ramsey numbers, \emph{SIAM J. Discrete Math.}, \textbf{32} (2018), 2217--2228.
        
    \bibitem{konig} D. K{\" o}nig, {\" U}ber Graphen und ihre Anwendung auf Determinantentheorie und Mengenlehre, \emph{Math. Ann.}, \textbf{77} (1916), 453--465.

    \bibitem{Pikhurko} O. Pikhurko, Size Ramsey numbers of stars versus 3-chromatic graphs, \emph{Combinatorica}, \textbf{21} (2001), 403--412.
        
    \bibitem{k} K. Tikhomirov, On bounded degree graphs with large size-Ramsey numbers, \emph{Combinatorica}, \textbf{44} (2024), 9--14.
        
    \bibitem {VZ} V. Vizing, On an estimate of the chromatic class of a $p$-graph, \emph{Diskret. Analiz.}, {\bf 3}, 25--30 (1964).
    
    \bibitem{Zhang} K. Zhang, A note on the size Ramsey number for stars, \emph{J. Comb. Math. Comb. Comput.}, \textbf{11} (1992), 209--214.	

  \end{thebibliography}
\end{document}